\documentclass[12pt]{amsart}

\usepackage{CJKutf8}
\usepackage{amssymb}
\usepackage{eucal}
\usepackage{color}

\usepackage{multirow,bigdelim}
\usepackage{amsfonts}
\usepackage{dsfont}
\usepackage{amsmath}
\usepackage{mathrsfs}
\usepackage{todonotes}
\usepackage{hyperref}
\usepackage{multimedia}
\usepackage{graphicx}
\usepackage{indentfirst}
\usepackage{bm}
\usepackage{amsthm}
\usepackage{mathrsfs}
\usepackage{latexsym}
\usepackage{amsmath}
\usepackage{amssymb}
\usepackage{microtype}
\usepackage{relsize}
\usepackage{hyperref}

\allowdisplaybreaks[4]

\DeclareSymbolFont{SY}{U}{psy}{m}{n}
\DeclareMathSymbol{\emptyset}{\mathord}{SY}{'306}

\theoremstyle{plain}

\newtheorem{thm}{Theorem}[section]

\newtheorem{lem}[thm]{Lemma}

\theoremstyle{definition}
\newtheorem{defn}[thm]{Definition}
\newtheorem{rem}[thm]{Remark}

\numberwithin{equation}{section}

\def\beq{\begin{eqnarray}}
\def\eeq{\end{eqnarray}}
\def\beqa{\begin{eqnarray*}}
\def\eeqa{\end{eqnarray*}}

\begin{document}
\title{The Cowen-Douglas operators with strongly flag structure}

\author{Yufang Xie and Kui Ji$^{*}$}
\curraddr[Y. Xie, K. Ji ]{School of Mathematical Sciences, Hebei Normal University,
Shijiazhuang, Hebei 050016, China}

\email[Y. Xie]{xieyufangmath@outlook.com}
\email[K. Ji]{jikui@hebtu.edu.cn, jikuikui@163.com}

\begin{abstract}
Denote $\mathcal{FB}_{n}(\Omega)$ as the collection of operators possessing a flag structure in the Cowen-Douglas class $\mathcal{B}_{n}(\Omega)$, and all the irreducible homogeneous operators in $\mathcal{B}_{n}(\Omega)$ belong to this class. G. Misra et al. pointed out in \cite{JJKM} that the unitary invariants of this class of operators include the curvature and the second fundamental form of the corresponding line bundle. In terms of the invariants, it is more tractable compared to general operators in $\mathcal{B}_{n}(\Omega)$. A subclass of $\mathcal{FB}_{n}(\Omega)$, denoted by $\mathcal{CFB}_{n}(\Omega)$, was proven to be norm dense in $\mathcal{B}_{n}(\Omega)$ in \cite{JJ}. In this paper, we introduce a smaller subclass of $\mathcal{FB}_{n}(\Omega)$ which possesses a strongly flag structure, and for which the curvature and the second fundamental form of the associated line bundle is a complete set of unitary invariants. And we notice that this class of operators is norm dense in $\mathcal{B}_{n}(\Omega)$ up to similarity. On this basis, we have completed the similar classification of a large class of operators with flag structure, which reduces the number of the similarity invariants in \cite{JKSX} from $\frac{n(n-1)}{2}+1$ to $n$. Furthermore, we also get a complete characterization of weakly homogeneous operators with high index and flag structure.
\end{abstract}

\maketitle

\section{Introduction}
Let $\Omega$ be a connected open subset of the complex plane $\mathbb{C}$ and $\mathcal{H}$ be a complex separable Hilbert space. Let $\mathcal{L}(\mathcal{H})$ denote the algebra of all bounded linear operators on a complex separable Hilbert space $\mathcal{H}$. In 1978, M. J. Cowen and R. G. Douglas defined an important class of the operators , denoted by $\mathcal{B}_{n}(\Omega)$. For any $T\in\mathcal{B}_{n}(\Omega)$, $E_{T}$ denotes the Hermitian holomorphic bundle with fibres $\ker(T-w)$, $w\in\Omega$. The operators in this class are determined by their eigenvectors, corresponding to eigenvalues from $\Omega$. It indicates that there exists a one to one correspondence between the unitary equivalence class of the operators $T$ and the local equivalence classes of the Hermitian holomorphic vector bundles $E_{T}$ determined by them. Consequently, the invariants of $E_{T}$ serve as unitary invariants for the operator $T$. As shown in \cite{CD}, when $n$ is equal to 1, the curvature of $E_{T}$ is a completely unitary invariant of $T$. Subsequently, in \cite{Misra}, G. Misra used the curvature of holomorphic vector bundle to describe the homogeneous operator with index 1, which is unitarily equivalent to the adjoint of the multiplication operator $M_{z}$ on the weighted Bergman space. When $n$ is greater than 1, however, the equality of curvatures alone does not determine the unitary equivalence of the operators. 
Therefore, it is worthy to find a complete set of tractable unitary invariants.

As stated in the book \cite{JW}, for an operator $T$ in the Cowen-Douglas class $\mathcal{B}_{n}(\Omega)$, $T$ can be written as an $n\times n$ upper-triangular matrix $T=\big((T_{i,j})\big)_{n\times n}$, where $T_{i,i} $ is in $\mathcal{B}_{1}(\Omega)$ for all $1\le i\le n$ and non-zero elements $T_{i,j}$ are bounded linear operators. What's more, $T$ is said to possess a flag structure if $T_{i,i}T_{i,i+1}=T_{i,i+1}T_{i+1,i+1}$ for $1\le i\le n-1$. C. L. Jiang, D. K. Keshari, G. Misra and the second author, in \cite{JJKM}, introduced the collection of operators possessing a flag structure in $\mathcal{B}_{n}(\Omega)$, denoted by $\mathcal{FB}_{n}(\Omega)$.
\begin{defn}\cite{JJKM}\label{FBn}
$\mathcal{FB}_{n}(\Omega)$ is the collection of all bounded linear operators $T$ defined on some complex separable Hilbert space $\mathcal{H}=\mathcal{H}_{1}\oplus\cdots\oplus\mathcal{H}_{n}$, which are written as an $n\times n$ upper-triangular matrix
$$T=\begin{pmatrix}
T_{1,1}&T_{1,2}&T_{1,3}&\cdots&T_{1,n}\\
0&T_{2,2}&T_{2,3}&\cdots&T_{2,n}\\
\vdots&\ddots&\ddots&\ddots&\vdots\\
0&\cdots&0&T_{n-1,n-1}&T_{n-1,n}\\
0&\cdots&\cdots&0&T_{n,n}
\end{pmatrix},$$
where the operator $T_{i,i}:\mathcal{H}_{i}\rightarrow\mathcal{H}_{i}$, defined on a complex separable Hilbert space $\mathcal{H}_{i}$, $1\le i\le n$, is assumed to be in $\mathcal{B}_{1}(\Omega)$ and $T_{i,i+1}:\mathcal{H}_{i+1}\rightarrow\mathcal{H}_{i}$, is assumed to be a non-zero intertwining bounded operator, namely, $T_{i,i}T_{i,i+1}=T_{i,i+1}T_{i+1,i+1},\ 1\le i\le n-1$.
\end{defn}

It is shown that when $T$ is in $\mathcal{FB}_{2}(\Omega)$, we can find a holomorphic frame $\gamma=\{\gamma_{0},\gamma_{1}\}$ where $\gamma_{0}(w)$ and $\frac{\partial}{\partial w}\gamma_{0}(w)-\gamma_{1}(w)$ are orthogonal for all $w$ in $\Omega$. Consequently, the explicit description of an operator $T\in\mathcal{FB}_{2}(\Omega)$ can be given,
which determines that the holomorphic change of frame for the vector bundle $E_{T}$, preserving the orthogonality relation between $\gamma_{0}(w)$ and $\frac{\partial}{\partial w}\gamma_{0}(w)-\gamma_{1}(w)$, must be of the form $\begin{pmatrix}\phi&\phi'\\0&\phi\end{pmatrix}$(see \cite{JJKM}, Lemma 2.5).
The important consequence of this conclusion is that  the unitary operator intertwining $T$ with other operator in $\mathcal{FB}_{2}(\Omega)$ must be diagonal. On the basis of $n=2$ and using mathematical induction, it is also obtained that the unitary operator intertwining two operators in $\mathcal{FB}_{n}(\Omega)$ must be diagonal(see \cite{JJKM}, Theorem 3.5). This phenomenon is called as ``rigidity''. It means that if a diagonal unitary operator intertwining two operators, one of which possesses a flag structure, then other operator also possesses a flag structure. Conversely, if two operators possessing a flag structure are unitarily equivalent, then the unitary operator intertwining them must be a diagonal operator.
Moreover, if two operators $T\in \mathcal{L}(\mathcal{H}_{1}\oplus\mathcal{H}_{2}\oplus\cdots\oplus\mathcal{H}_{n}),\  \widetilde{T}\in\mathcal{L}(\widetilde{\mathcal{H}}_{1}\oplus\widetilde{\mathcal{H}}_{2}\oplus\cdots\oplus\widetilde{\mathcal{H}}_{n})$ which belong to the class $\mathcal{FB}_{n}(\Omega)$
are unitarily equivalent, then we can also find a digonal unitary operator intertwining $T|_{\mathcal{H}_{1}\oplus\mathcal{H}_{2}\oplus\cdots\oplus\mathcal{H}_{k}}$ and $\widetilde{T}|_{\widetilde{\mathcal{H}}_{1}\oplus\widetilde{\mathcal{H}}_{2}\oplus\cdots\oplus\widetilde{\mathcal{H}}_{k}}$ for any $1\le k\le n$. In fact, if an unitary operator $U=U_{1}\oplus U_{2}\oplus\cdots\oplus U_{n}$ intertwines $T$ and $\tilde{T}$, then we have $$(U_{1}\oplus U_{2}\oplus\cdots\oplus U_{k})T|_{\mathcal{H}_{1}\oplus\mathcal{H}_{2}\oplus\cdots\oplus\mathcal{H}_{k}}
=\widetilde{T}|_{\widetilde{\mathcal{H}}_{1}\oplus\widetilde{\mathcal{H}}_{2}\oplus\cdots\oplus\widetilde{\mathcal{H}}_{k}}
(U_{1}\oplus U_{2}\oplus\cdots\oplus U_{k}),$$
where each $U_{1}\oplus U_{2}\oplus\cdots\oplus U_{k}$ is an unitary operator for any $1\le k\le n$.
By the rigidity, a complete set of unitary invariants for operators in $\mathcal{FB}_{n}(\Omega)$ is obtained.

\begin{thm}\cite{JJKM}
For $T,\ \tilde{T}\in\mathcal{FB}_{n}(\Omega)$,
$$T\sim_{u}\tilde{T}\ \ \ \mbox{if\ and\ only\ if}\ \ \left\{\begin{smallmatrix}\mathcal{K}_{T_{n,n}}=\mathcal{K}_{\tilde{T}_{n,n}},\\
\theta_{i,i+1}(T)=\theta_{i,i+1}(\tilde{T})\\
\frac{\langle T_{i,j}(t_{j}),t_{i}\rangle}{\Vert t_{i}\Vert^{2}}=\frac{\langle \tilde{T}_{i,j}(\tilde{t}_{j}),\tilde{t}_{i}\rangle}{\Vert \tilde{t}_{i}\Vert^{2}}\\\end{smallmatrix}
\right\}.$$
\end{thm}

From the above theorem, the curvature $\mathcal{K}_{T_{n,n}}$ and the second fundamental forms $\theta_{i,i+1}(T)$ alone cannot describe the completely unitary invariants for operators in $\mathcal{FB}_{n}(\Omega)$ when $n$ is greater than 2. It refers to a set of $\frac{n(n-1)}{2}+1$ completely unitary invariants which are easy to compute. We should note that the operators in $\mathcal{FB}_{2}(\Omega)$ are very special. They not only possess a flag structure, but also only the curvature and the second fundamental form are used to characterize their completely unitary invariants. Therefore in this paper, we isolate this class of operators including $\mathcal{FB}_{2}(\Omega)$ and call them possess a strongly flag structure.

Unfortunately, however, an invertible operator can not induce an isometric bundle map. This is the reason why finding similar invariants is more tricker. So far, there have been only some partial results (cf. \cite{HJK}, \cite{Kwon1}, \cite{Tr}) and how to make use of the curvature to determine when two Cowen-Douglas operators are similar is still unclear.

The similarity of operators is also related to the problem of weak homogeneity of operators, which is an extension and generalization of homogeneous operators.
There has been rich results on the homogeneous operators such as \cite{BM1}, \cite{BM2}, \cite{BM}, \cite{Misra}, \cite{AM}, \cite{Wilkins}, \cite{AG}.
In contrast, few conclusions have been drawn about weakly homogeneous operators.
N. Zorboska determined that the operator $M^{*}_{z}$ on the reproducing kernel Hilbert space $(\mathcal{H},K_{(\gamma)})$ is a weakly homogeneous operator in $\mathcal{B}_{1}(\mathbb{D})$(see \cite{CM},\cite{NZ}, or \cite{Ghara}). 
Then given a sharp non-negative kernel, in \cite{Ghara}, S. Ghara characterized weakly homogeneous operators in $\mathcal{B}_{1}(\mathbb{D})$.
Furthermore, according to the results in \cite{JJKM}, S. Ghara found a class of weakly homogeneous operators in the class $\mathcal{FB}_{2}(\mathbb{D})$ and gave non-examples of weakly operators in the same class(see \cite{Ghara}). Recently, B. Z. Hou et al. in \cite{HJ} introduced the growth types of weighted Hardy spaces and proved that on a weighted Hardy spaces of polynomial growth, the multiplication operator $M_{z}$ is a weakly homogeneous operator.

The present paper is structured as follows. Firstly, we recall some fundamental concepts and basic results in section 2. Subsequently, we introduce a subclass of Cowen-Douglas operators $\mathcal{OFB}_{n}(\Omega)$, which is said to possess a strongly flag structure, and we demonstrate that this class is norm dense in $\mathcal{B}_{n}(\Omega)$ up to similarity.
During this process, we find that an operator $T\in\mathcal{FB}_{n}(\Omega)$ satisfying certain condition (denoted by Condition (A)) is similar to an operator $\tilde{T}$ whose non-zero elements only on the primary and secondary diagonals and are same with $T$. Moreover, the similarity invariants for a large class of operators in $\mathcal{FB}_{n}(\Omega)$ is given in section 4. Lastly, we obtain a sufficient and necessary condition for a large class of weakly homogeneous operators.

\section{Preliminaries}

In order to explain our main theorem clearly, we will introduce the following
notations and definitions.

\subsection{The class $\mathcal{B}_{n}(\Omega)$}
Let $n$ be a positive integer. The class of Cowen-Douglas operators with index $n$ introduced in \cite{CD} is defined as follows:
$$\begin{array}{lll}
\mathcal{B}_{n}(\Omega):=\{T\in\mathcal{L}(\mathcal{H}):
&(1)\,\,\Omega\subset\sigma(T):=\{w\in\mathbb{C}:\ T-w\ \mbox{is\ not\ invertible}\},\\
&(2)\,\,\mbox{ran}(T-w)=\mathcal{H}\ \mbox{for\ all}~w\in\Omega,\\
&(3)\,\,\bigvee_{w\in \Omega}\ker(T-w)=\mathcal{H},\\
&(4)\,\,\mbox{dim ker}(T-w)=n\ \mbox{for\ all}~w\in\Omega\}.
\end{array}$$
In the same paper(cf \cite{CD}), Cowen and Douglas showed that for an operator $T$ in $\mathcal{B}_{n}(\Omega)$, the mapping $z\mapsto\ker(T-w)$ defines a rank $n$ Hermitian holomorphic vector bundle $E_{T}$ on $\Omega$,
$$E_{T}=\{(w,x)\in\Omega\times\mathcal{H}:\ x\in \ker(T-w)\}\ \mbox{and}\ \pi:E_{T}\rightarrow\Omega,\ \mbox{where}\ \pi(w,x)=w.$$
The Gram matrix corresponding to a holomorphic frame $\{\gamma_{1},\cdots,\gamma_{n}\}$ for $E_{T}$ is defined by
$$h(w):=\big( \langle\gamma_{j}(w),\gamma_{i}(w)\rangle\big)_{i,j=1}^{n},\ w\in\Omega.$$
Then the definition of the curvature function is given by
$$\mathcal{K}(w):=\frac{\partial}{\partial\bar{w}}\big(h^{-1}\frac{\partial h}{\partial w}\big).$$
In particular, when $T\in\mathcal{B}_{1}(\Omega)$, the curvature function of $T$ can also be written as $$\mathcal{K}(w)=-\frac{\partial^{2}}{\partial \bar{w}\partial w}\log\Vert\gamma(w)\Vert^{2},$$
where $\gamma(w)$ is a non-vanishing holomorphic cross-section of the bundle $E_{T}$. For any $C^{\infty}$-bundle map $\phi$ on the Hermitian holomorphic vector bundle $E$, and given a holomorphic frame $\gamma$, we have
\begin{enumerate}
\item $\phi_{\bar{w}}(\gamma)=\frac{\partial}{\partial\bar{w}}(\phi(\gamma))$,

\item $\phi_{w}(\gamma)=\frac{\partial}{\partial w}(\phi(\gamma))+[h^{-1}\frac{\partial}{\partial w}h,\phi(\gamma)]$.
\end{enumerate}
The covariant partial derivatives of curvature, denoted by $\mathcal{K}_{w^{i},\bar{w}^{j}},\ i,j\in\mathbb{N}\cup\{0\}$, can be obtained by using the inductive formulas above due to  the fact that curvature can be viewed as a bundle map. Then M. J. Cowen and R. G. Douglas gave the answer to the problem of unitary classification of operators in terms of curvature and its covariant partial derivatives.
\begin{thm}\cite{CD}
Let $E_{T}$ and $E_{S}$ be two Hermitian holomorphic vector bundles induced by two Cowen-Douglas operators $T$ and $S$ with index $n$, respectively. Then $E_{T}\sim_{u}E_{S}$ if and only if there exists an isometry $V:E_{T}\rightarrow E_{S}$ such that
$$V\mathcal{K}_{T,w^{i},\bar{w}^{j}}=\mathcal{K}_{S,w^{i},\bar{w}^{j}}V,\ 0\le i\le j\le i+j\le n,\ (i,j)\ne(0,n),(n,0).$$
\end{thm}

\subsection{The second fundamental form}
The computation of the second fundamental form is given below following \cite{JJKM}. Let $E_{T}$ be the vector bundle of rank $2$ corresponding to the operator $T=\begin{pmatrix}T_{1,1}&T_{1,2}\\0&T_{2,2}\end{pmatrix}$ in $\mathcal{FB}_{2}(\Omega)$ and $E_{1}$ be the line bundle corresponding to the operator $T_{1,1}$. Suppose $\{\gamma_{1},\gamma_{2}\}$ be a holomorphic frame for $E_{T}$ such that $\gamma_{1}$ and $\frac{\partial}{\partial w}\gamma_{1}-\gamma_{2}$ are orthogonal. Denote by $\{e_{1},e_{2}\}$ an orthogonal frame from the holomorphic frame $\{\gamma_{1},\gamma_{2}\}$ by the Gram-Schmidt process: Set $h=\langle\gamma_{1},\gamma_{1}\rangle$, then
$$e_{1}=h^{-\frac{1}{2}}\gamma_{1},\ \ e_{2}=\frac{\gamma_{2}-\frac{\gamma_{1}\langle\gamma_{2},\gamma_{1}\rangle}{\Vert \gamma_{1}\Vert^{2}}}{(\Vert\gamma_{2}\Vert^{2}-\frac{|\langle\gamma_{2},\gamma_{1}\rangle|^{2}}{\Vert\gamma_{1}\Vert^{2}})^{\frac{1}{2}}}
.$$
And one can obtain
$$\bar{\partial}e_{1}=-\frac{1}{2}h^{-\frac{3}{2}}\bar{\partial}h\gamma_{1}=-\frac{1}{2}h^{-1}\bar{\partial}he_{1}
=-\frac{1}{2}\bar{\partial}(\log h)e_{1},$$
$$\bar{\partial}e_{2}=-h^{\frac{1}{2}}\frac{\bar{\partial}(h^{-1}\langle\gamma_{2},\gamma_{1}\rangle)}
{(\Vert\gamma_{2}\Vert^{2}-\frac{|\langle\gamma_{2},\gamma_{1}\rangle|^{2}}{\Vert\gamma_{1}\Vert^{2}})^{\frac{1}{2}}}e_{1}--\frac{\bar{\partial}(\Vert\gamma_{2}\Vert^{2}-\frac{\langle\gamma_{2},\gamma_{1}\rangle}{\Vert\gamma_{1}\Vert^{2}})}{(\Vert\gamma_{2}\Vert^{2}-\frac{|\langle\gamma_{2},\gamma_{1}\rangle|^{2}}{\Vert\gamma_{1}\Vert^{2}})^{\frac{1}{2}}}e_{2}.$$

The canonical hermitian connection $D$ for the vector bundle $E_{T}$ is given in terms of $e_{1}$ and $e_{2}$ by the formula:
\begin{equation*}
\begin{aligned}
De_{1}&=D^{1,0}e_{1}+D^{0,1}e_{1}\\
&=\alpha_{11}e_{1}+\alpha_{21}e_{2}+\bar{\partial}e_{1}\\
&=\Big(\alpha_{11}-\frac{1}{2}\bar{\partial}(\log h)\Big)e_{1}+\alpha_{21}e_{2}\\
&=\theta_{11}e_{1}+\theta_{21}e_{2},
\end{aligned}
\end{equation*}
where $\alpha_{11}$ and $\alpha_{21}$ are $(1,0)$ forms to be determined. Similarly, we can also obtain
\begin{equation*}
\begin{aligned}
De_{2}&=D^{1,0}e_{1}+D^{0,1}e_{1}\\
&=\alpha_{12}e_{1}+\alpha_{22}e_{2}+\bar{\partial}e_{2}\\
&=\Big(\alpha_{12}-h^{\frac{1}{2}}\frac{\bar{\partial}(h^{-1}\langle\gamma_{2},\gamma_{1}\rangle)}
{(\Vert\gamma_{2}\Vert^{2}-\frac{|\langle\gamma_{2},\gamma_{1}\rangle|^{2}}{\Vert\gamma_{1}\Vert^{2}})^{\frac{1}{2}}}\Big)e_{1}
+\Big(\alpha_{22}-\frac{1}{2}\frac{\bar{\partial}(\Vert\gamma_{2}\Vert^{2}-\frac{\langle\gamma_{2},\gamma_{1}\rangle}{\Vert\gamma_{1}\Vert^{2}})}{(\Vert\gamma_{2}\Vert^{2}-\frac{|\langle\gamma_{2},\gamma_{1}\rangle|^{2}}{\Vert\gamma_{1}\Vert^{2}})}\Big)e_{2}\\
&=\theta_{12}e_{1}+\theta_{22}e_{2},
\end{aligned}
\end{equation*}
where $\alpha_{12}$ and $\alpha_{22}$ are $(1,0)$ forms to be determined. Following the compatibility of the connection with the Hermitian metric, one can see that
$$\langle De_{i},e_{j}\rangle+\langle e_{i},De_{j}\rangle=\theta_{ji}+\bar{\theta}_{ij}=0,\ \mbox{for}\ \ 1\le i,j\le2.$$
Then we get $\alpha_{11}=\frac{1}{2}\partial(\log h)$, $\alpha_{12}=0$, $\alpha_{21}=h^{\frac{1}{2}}\frac{\bar{\partial}(h^{-1}\langle\gamma_{2},\gamma_{1}\rangle)}
{(\Vert\gamma_{2}\Vert^{2}-\frac{|\langle\gamma_{2},\gamma_{1}\rangle|^{2}}{\Vert\gamma_{1}\Vert^{2}})^{\frac{1}{2}}}$ and $\alpha_{22}=\frac{1}{2}\frac{\bar{\partial}(\Vert\gamma_{2}\Vert^{2}-\frac{\langle\gamma_{2},\gamma_{1}\rangle}{\Vert\gamma_{1}\Vert^{2}})}{(\Vert\gamma_{2}\Vert^{2}-\frac{|\langle\gamma_{2},\gamma_{1}\rangle|^{2}}{\Vert\gamma_{1}\Vert^{2}})}$.
Therefore, the second fundamental form for the inclusion $E_{1}\subset E_{T}$ is given by the formula:
$$\theta_{12}=-h^{\frac{1}{2}}\frac{\bar{\partial}(h^{-1}\langle\gamma_{2},\gamma_{1}\rangle)}
{(\Vert\gamma_{2}\Vert^{2}-\frac{|\langle\gamma_{2},\gamma_{1}\rangle|^{2}}{\Vert\gamma_{1}\Vert^{2}})^{\frac{1}{2}}}
=-\frac{\frac{\partial^{2}}{\partial\bar{w}\partial w}\log h d\bar{w}}{\big(\frac{\Vert t_{1}\Vert^{2}}{\Vert\gamma_{1}\Vert^{2}}+\frac{\partial^{2}}{\partial\bar{w}\partial w}\log h\big)^{\frac{1}{2}}},$$
where $t_{1}=\partial\gamma_{1}-\gamma_{2}$. If $t_{2}$ is a non-vanishing holomorphic section of the vector bundle $E_{2}$ corresponding to the operator $T_{2,2}$, then $T_{1,2}(t_{2})$ is a holomorphic frame of $E_{1}$. Then the second fundamental form $\theta_{12}$ of the inclusion $E_{1}\subset E_{T}$ equals to
$$-\frac{\frac{\partial^{2}}{\partial \bar{w}\partial w}\log\Vert T_{1,2}(t_{2})\Vert^{2}d\bar{w}}{\Big(\frac{\Vert T_{1,2}(t_{2})\Vert^{2}}{\Vert t_{2}\Vert^{2}}+\frac{\partial^{2}}{\partial\bar{w}\partial w}\log\Vert T_{1,2}(t_{2})\Vert^{2}\Big)^{\frac{1}{2}}}.$$

For $T\in\mathcal{FB}_{n}(\Omega)$ with the upper-triangular matrix form given by Definition \ref{FBn}, the corresponding second fundamental form $\theta_{i,i+1}(T)$ of $E_{T_{i,i}}$ is given by the formula:
$$\theta_{i,i+1}(T)(w)=\frac{\mathcal{K}_{T_{i,i}}(w)d\bar{w}}{\Big(\frac{\Vert T_{i,i+1}t_{i+1}(w)\Vert^{2}}{\Vert t_{i+1}(w)\Vert^{2}}-\mathcal{K}_{T_{i,i}}(w)\Big)^{\frac{1}{2}}},$$
where $t_{i+1}$ is a non-vanishing section of $E_{T_{i+1,i+1}}$. For two arbitrary operators $T$ and $\tilde{T}$ in $\mathcal{FB}_{n}(\Omega)$ with $\mathcal{K}_{T_{i,i}}=\mathcal{K}_{\tilde{T}_{i,i}}$, we have
$$\theta_{i,i+1}(T)(w)=\theta_{i,i+1}(\tilde{T})(w)\ \mbox{if\ and\ only\ if}\ \frac{\Vert T_{i,i+1}(t_{i+1}(w))\Vert^{2}}{\Vert t_{i+1}(w)\Vert^{2}}=\frac{\Vert\tilde{T}_{i,i+1}(\tilde{t}_{i+1}(w))\Vert^{2}}{\Vert \tilde{t}_{i+1}(w)\Vert^{2}}.$$
Consequently, we can also use $\frac{\Vert T_{i,i+1}(t_{i+1}(w))\Vert^{2}}{\Vert t_{i+1}(w)\Vert^{2}}$ instead of  $\theta_{i,i+1}(T)$.

\subsection{Reproducing kernel Hilbert spaces}
Let $\mathcal{H}$ be a separable Hilbert space consisting of $\mathbb{C}$-valued holomorphic functions over $\Omega$ where $\Omega$ is a subset in $\mathbb{C}$. The function $K:\Omega\times\Omega\rightarrow\mathbb{C}$ is called a reproducing kernel of $\mathcal{H}$ if for every $w\in\Omega$,
\begin{enumerate}
\item $K(\cdot,w)$ as a holomorphic function belongs to $\mathcal{H}$.
\item The reproducing property: for every $ f\in\mathcal{H}$,
    $$\langle f,K(\cdot,w)\rangle_{\mathcal{H}}= f(w).$$

\end{enumerate}
Equivalently, if $\{e_{n}\}_{n=0}^{\infty}$ is an orthonormal basis of $\mathcal{H}$, then the reproducing kernel function $K$ can also be defined by $$K(z,w)=\sum\limits_{n=0}^{\infty}e_{n}(z)e_{n}(w)^{*}.$$
Clearly, $\mathcal{H}=\overline{\mbox{span}}\{K(\cdot,\bar{w}):\ w\in\Omega\}$. We always assume that the kernel $K(z,w)$ is sesqui-analytic, that is, it is holomorphic in $z$ and anti-holomorphic in $w$.

A function $K:\Omega\times\Omega\rightarrow\mathbb{C}$ is said to be a positive definite kernel if for any subset $\{w_{1},\cdots,w_{n}\}$ of $\Omega$ and for any complex numbers $a_{1},\cdots,a_{n}$,
$$\sum\limits_{i,j=1}^{n}\bar{a_{i}}a_{j}K(w_{i},w_{j})\ge0.$$
Generally, a reproducing kernel $K(z,w)=\sum\limits_{n=0}^{\infty}a_{n}z^{n}\bar{w}^{n}$ is positive if $a_{n}>0$ for all $n\ge0$. And we could know $e_{n}(z)=\sqrt{a_{n}}z^{n}, n\ge0$, is an orthonormal basis of the Hilbert space.

\subsection{Multipliers}Denote by $\mbox{Hol}(\mathbb{D})$ the space of all analytic functions on the unit open disc $\mathbb{D}$. For two Hilbert spaces $\mathcal{H}_{1}$ and $\mathcal{H}_{2}$ consisting of holomorphic functions on the unit disc $\mathbb{D}$, the multiplier algebra Mult$(\mathcal{H}_{1},\mathcal{H}_{2})$ is defined as
$$\mbox{Mult}(\mathcal{H}_{1},\mathcal{H}_{2}):=\{\psi\in \mbox{Hol}(\mathbb{D}):\ \psi f\in\mathcal{H}_{2}\ \mbox{whenever}\ f\in\mathcal{H}_{1}\}.$$
When $\mathcal{H}_{1}=\mathcal{H}_{2}$, we write $\mbox{Mult}(\mathcal{H}_{1})$ instead of Mult$(\mathcal{H}_{1},\mathcal{H}_{2})$.

Using the closed graph theorem, it is easy to see that $\psi\in\mbox{Mult}(\mathcal{H}_{1},\mathcal{H}_{2})$ if and only if the multiplication operator $M_{\psi}$ is bounded from $\mathcal{H}_{1}$ to $\mathcal{H}_{2}$. Generally speaking, $\mbox{Mult}(\mathcal{H}_{1},\mathcal{H}_{2})$, or $\mbox{Mult}(\mathcal{H}_{1})$, is a proper subset of $\mbox{H}^{\infty}(\mathbb{D})$, where $\mbox{H}^{\infty}(\mathbb{D})$ is the algebra of all bounded holomorphic functions on the unit disc $\mathbb{D}$, such as the multiplier algebra on the Dirichlet space. And there are also some spaces whose multiplier algebras are equal to $\mbox{H}^{\infty}(\mathbb{D})$, such as the Hardy space and the Bergman space.

\begin{lem}\cite{Richter},\cite{shields}
Let $M_{i,z}$ be the bounded multiplication operator on the reproducing kernel Hilbert space $\mathcal{H}_{K_{i}},\ i=1,2$. Then an operator $T\in\mathcal{L}(\mathcal{H}_{K_{1}},\mathcal{H}_{K_{2}})$ satisfies $TM_{1,z}=M_{2,z}T$ if and only if there is a function $\psi\in\mbox{Mult}(\mathcal{H}_{K_{1}},\mathcal{H}_{K_{2}})$ such that $T=M_{\psi}$.
\end{lem}

\section{Strongly flag structure}
In \cite{JJKM}, we have learnt that an operator in $\mathcal{B}_{n}(\Omega)$ with a flag structure refers to a set of $\frac{n(n-1)}{2}+1$ completely unitary invariants including the curvature and the second fundamental form. In particular, there exist operators that one can describe their completely unitary invariants in terms of the curvature and the second fundamental form. The proof of the following result is similar to Theorem 3.6 in \cite{JJKM}. For the completeness of the paper, we give the following proof.
\begin{thm}
Any two operators $T=\big((T_{i,j})\big)_{n\times n}$ and $\tilde{T}=\big((\tilde{T}_{i,j})\big)_{n\times n}$ in $\mathcal{FB}_{n}(\Omega)$, where $T_{i,j},\ \tilde{T}_{i,j}=0$ if $j-i\ge2$, are unitarily equivalent if and only if the following conditions hold:
\begin{enumerate}
\item $\mathcal{K}_{T_{n,n}}=\mathcal{K}_{\tilde{T}_{n,n}}$,

\item $\frac{\Vert t_{i}\Vert}{\Vert \tilde{t}_{i}\Vert}=\frac{\Vert t_{i+1}\Vert}{\Vert \tilde{t}_{i+1}\Vert}$, where $t_{i},\tilde{t}_{i}$, $1\le i\le n$, is a non-vanishing holomorphic section of $E_{T_{i,i}},\ E_{\tilde{T}_{i,i}}$, respectively, and $t_{i}(w)=T_{i,i+1}t_{i+1}(w),\ \tilde{t}_{i}(w)=\tilde{T}_{i,i+1}\tilde{t}_{i+1}(w)$ for $1\le i\le n-1$.
\end{enumerate}
\end{thm}
\begin{proof}
Suppose there exists an unitary operator $U$ such that $UT=\tilde{T}U$. By Theorem 3.5 in \cite{JJKM}, this $U$ must be diagonal, that is, $U=U_{1}\oplus\cdots\oplus U_{n}$, for some unitary operators $U_{i}$, $1\le i\le n$. Following the relationship $UT=\tilde{T}U$, we can see that $U_{i}T_{i,i}=\tilde{T}_{i,i}U_{i},\ 1\le i\le n$ and $U_{i}T_{i,i+1}=\tilde{T}_{i,i+1}U_{i},\ 1\le i\le n-1$. Then there exists some non-zero holomorphic function $\phi$ such that $$U_{i}(t_{i}(w))=\phi(w)\tilde{t}_{i}(w),\ w\in\Omega,\ 1\le i\le n.$$
Thus we can obtain
\begin{equation*}
\begin{aligned}
\mathcal{K}_{T_{n,n}}&=-\frac{\partial^{2}}{\partial\bar{w}\partial w}\log\Vert t_{n}(w)\Vert^{2}\\
&=-\frac{\partial^{2}}{\partial\bar{w}\partial w}\log\Vert U_{n}(t_{n}(w))\Vert^{2}\\
&=-\frac{\partial^{2}}{\partial\bar{w}\partial w}\log\Vert \phi(w)\tilde{t}_{n}(w)\Vert^{2}\\
&=\mathcal{K}_{\tilde{T}_{n,n}},
\end{aligned}
\end{equation*}
and $\frac{\Vert t_{i+1}\Vert}{\Vert t_{i}\Vert}=\frac{\Vert U_{i+1}t_{i+1}\Vert}{\Vert U_{i}t_{i}\Vert}=\frac{\Vert\phi\tilde{t}_{i+1}\Vert}{\Vert\phi\tilde{t}_{i}\Vert}=\frac{\Vert \tilde{t}_{i+1}\Vert}{\Vert \tilde{t}_{i}\Vert}$ for $1\le i\le n-1$.

On the other hand, assume that $T$ and $\tilde{T}$ are operators satisfying the conditions. Following the condition $\mathcal{K}_{T_{n,n}}=\mathcal{K}_{\tilde{T}_{n,n}}$,  there exists some non-zero holomorphic function $\phi$ defined on $\Omega$ such that
$\Vert t_{n}(w)\Vert=|\phi(w)|\Vert\tilde{t}_{n}(w)\Vert$. Then by the relationship $\frac{\Vert t_{i}\Vert}{\Vert \tilde{t}_{i}\Vert}=\frac{\Vert t_{i+1}\Vert}{\Vert \tilde{t}_{i+1}\Vert}$, $1\le i\le n-1$, we can obtain
$$\Vert t_{i}(w)\Vert=|\phi(w)|\Vert\tilde{t}_{i}(w)\Vert,$$ for $1\le i\le n$. We next define $U_{i}:\mathcal{H}_{i}\rightarrow\tilde{\mathcal{H}}_{i},\ 1\le i\le n$, as $$U_{i}(t_{i}(w))=\phi(w)\tilde{t}_{i}(w),\ w\in\Omega, $$
and extend to the linear span of those vectors. Obviously, for $1\le i\le n$, $U_{i}$ is an isometry from $\mathcal{H}_{i}$ to $\tilde{\mathcal{H}}_{i}$ and satisfies $U_{i}T_{i,i}=\tilde{T}_{i,i}U_{i}$ by a simple calculation. Then due to the range of non-zero operator intertwining two Cowen-Douglas operators with index 1 is dense, $U_{i},\ 1\le i\le n$, is an unitary operator. And it is easy to check that $U_{i}T_{i,i+1}=\tilde{T}_{i,i+1}U_{i}$ for $1\le i\le n-1$. Set $$U=\begin{pmatrix}U_{1}&0&0&\cdots&0\\ 0&U_{2}&0&\cdots&0\\ \vdots&\ddots&\ddots&\ddots&\vdots\\
0&\cdots&0&U_{n-1}&0\\
0&\cdots&\cdots&0&U_{n}\end{pmatrix},$$
then $U$ is an unitary operator intertwining $T$ and $\tilde{T}$.
\end{proof}

We isolate the operators in the above theorem and introduce a new class of Cowen-Douglas operators.
\begin{defn}\label{OFBn}
A Cowen-Douglas operator $T$ with index $n$ is said to be in the class $\mathcal{OFB}_{n}(\Omega)$, if $T\in\mathcal{FB}_{n}(\Omega)$ and $T$ has the following form on $\mathcal{H}=\mathcal{H}_{1}\oplus\cdots\oplus\mathcal{H}_{n}$,
$$T=\begin{pmatrix}
T_{1,1}&T_{1,2}&0&\cdots&0\\
0&T_{2,2}&T_{2,3}&\cdots&0\\
\vdots&\ddots&\ddots&\ddots&\vdots\\
0&\cdots&0&T_{n-1,n-1}&T_{n-1,n}\\
0&\cdots&\cdots&0&T_{n,n}
\end{pmatrix}.$$
\end{defn}
Obviously, this class of operators possesses a flag structure. Furthermore, an operator $T$ in $\mathcal{OFB}_{n}(\Omega)$ is said to \textbf{possess a strongly flag structure}, that is, the curvature together with the second fundamental form is a complete set of unitary invariants.
The operators in $\mathcal{FB}_{n}(\Omega)$ possess a flag structure, particularly when $n=2$, in which case they possess a strongly flag structure. Therefore this definition makes sense.

The following is the main result of this section.
\begin{thm}\label{main 1}
For an arbitrary operator $T$ in $\mathcal{B}_{n}(\Omega)$, there exist a sequence of operators $\{T_{n}\}_{n=1}^{\infty}\subseteq\mathcal{OFB}_{n}(\Omega)$ and a sequence of bounded and invertible operators $\{X_{n}\}_{n=1}^{\infty}$ such that $$\Vert T-X_{n}T_{n}X_{n}^{-1}\Vert\rightarrow0.$$
\end{thm}

To prove our main theorem, we need to the following concepts and lemmas.

Recall that the commutant $\{T\}'$ of $T$ on $\mathcal{H}$ is the set of operators in $\mathcal{L}(\mathcal{H})$ commuting with $T$.
\begin{defn}\cite{JJ}\label{CFBn}
A Cowen-Douglas operator $T$ with index $n$ is said to be in $\mathcal{CFB}_{n}(\Omega)$, if $T$ satisfies the following properties:
\begin{enumerate}
\item $T$ can be written as an $n\times n$ upper-triangular matrix form under a topological direct decomposition of $\mathcal{H}$ and diag$\{T\}:=T_{1,1}\dotplus T_{2,2}\dotplus\cdots\dotplus T_{n,n}\in\{T\}'$. Furthermore, each entry$$T_{i,j}=\phi_{i,j}(T_{i,i})T _{i,i+1}T_{i+1,i+2}\cdots T_{j-1,j}$$where $\phi_{i,j}(T_{i,i})\in\{T_{i,i}\}'$ for $1\le i\le n$;
\item $T_{i,i},T_{i+1,i+1}$ satisfy the Property($\mbox{H}$), that is, $\ker\tau_{T_{i,i},T_{i+1,i+1}}\cap \mbox{ran}\tau_{T_{i,i},T_{i+1,i+1}}=\{0\}$, where $\tau_{T_{i,i},T_{i+1,i+1}}(X)=T_{i,i}X-XT_{i+1,i+1}$ for $X\in\mathcal{L}(\mathcal{H})$ and $1\le i\le n-1$.
\end{enumerate}
\end{defn}

\begin{lem}\cite{JJ}
$\mathcal{CFB}_{n}(\Omega)$ is norm dense in $B_{n}(\Omega)$.
\end{lem}

Next, we will investigate the similarity problem between operators in $\mathcal{FB}_{n}(\Omega)$ and $\mathcal{OFB}_{n}(\Omega)$. Prior to that, for the sake of convenience, we present the following definition.
\begin{defn}[\textbf{Condition (A)}]
A Cowen-Douglas operator $T=\big((T_{i,j})\big)_{n\times n}\in\mathcal{FB}_{n}(\Omega)$  is said to satisfy the Condition (A), if  $T_{i,j}=\phi_{i,j}(T_{i,i})T_{i,i+1}T_{i+1,i+2}\cdots T_{j-1,j}$ where $\phi_{i,j}(T_{i,i})\in\{T_{i,i}\}'$ for $1\le i\le n$.
\end{defn}

\begin{lem}\cite{JJKM}\label{J.J.K.M}
If $X$ is an invertible operator intertwining two operators $T$ and $\tilde{T}$ from $\mathcal{FB}_{n}(\Omega)$, then $X$ and $X^{-1}$ are upper triangular.
\end{lem}

When $n$ is equal to 3, we establish the following conclusion.

\begin{lem}\label{X.}
Let $T=\begin{pmatrix}
T_{1,1}&T_{1,2}&0\\
0&T_{2,2}&T_{2,3}\\
0&0&T_{3,3}
\end{pmatrix}\ \mbox{and}\ \widetilde{T}=\begin{pmatrix}
T_{1,1}&T_{1,2}&T_{1,3}\\
0&T_{2,2}&T_{2,3}\\
0&0&T_{3,3}
\end{pmatrix}$ be two operators in $\mathcal{FB}_{3}(\Omega)$ satisfying the Condition (A).
Then there exists a bounded operator $K$ such that $X=\textbf{I}+K$ is invertible, where $\textbf{I}$ means the identity operator, and $XT=\widetilde{T}X$.
\end{lem}

\begin{proof}
Set $X:=\textbf{I}+K$, where $$K=\begin{pmatrix}
0&K_{1,2}&K_{1,3}\\
0&0&K_{2,3}\\
0&0&0
\end{pmatrix}.$$
We consider the following equation
\begin{equation}
\begin{pmatrix}
\textbf{I}&K_{1,2}&K_{1,3}\\
0&\textbf{I}&K_{2,3}\\
0&0&\textbf{I}
\end{pmatrix}\begin{pmatrix}
T_{1,1}&T_{1,2}&0\\
0&T_{2,2}&T_{2,3}\\
0&0&T_{3,3}
\end{pmatrix}
=\begin{pmatrix}
T_{1,1}&T_{1,2}&T_{1,3}\\
0&T_{2,2}&T_{2,3}\\
0&0&T_{3,3}
\end{pmatrix}\begin{pmatrix}
\textbf{I}&K_{1,2}&K_{1,3}\\
0&\textbf{I}&K_{2,3}\\
0&0&\textbf{I}
\end{pmatrix}.
\end{equation}
By comparing the elements in $(i,j)$ position, we can get
\begin{equation}
\left\{
\begin{aligned}
T_{1,2}+K_{1,2}T_{2,2}&=T_{1,1}K_{1,2}+T_{1,2},\\
T_{2,3}+K_{2,3}T_{3,3}&=T_{2,2}K_{2,3}+T_{2,3},\\
K_{1,2}T_{2,3}+K_{1,3}T_{3,3}&=T_{1,1}K_{1,3}+T_{1,2}K_{2,3}+T_{1,3}.
\end{aligned}
\right.
\end{equation}
Let $K_{1,3}=T_{1,3},K_{2,3}=T_{2,3}$, then we can see $K_{1,3}T_{3,3}=T_{1,1}K_{1,3}$ and $K_{2,3}T_{3,3}=T_{2,2}K_{2,3}$ by a simple calculation. Furthermore, we will obtain $K_{1,2}=T_{1,2}+\phi_{1,3}(T_{1,1})T_{1,2}$ which satisfies $K_{1,2}T_{2,2}=T_{1,1}K_{1,2}$.

To sum up, we find $K$ of the form $$K=\begin{pmatrix}
0&T_{1,2}+
\phi_{1,3}(T_{1,1})T_{1,2}&T_{1,3}\\
0&0&T_{2,3}\\
0&0&0
\end{pmatrix},$$
and it satisfies $(\textbf{I}+K)T=\widetilde{T}(\textbf{I}+K)$.
\end{proof}

The following we will use mathematical induction to come up with a conclusion of rank with $n$.
\begin{lem} \label{mainlemma}
Let $T=\big((T_{i,j})\big)_{n\times n}$ be an operator in $\mathcal{OFB}_{n}(\Omega)$ and  $\tilde{T}=\big((T_{i,j})\big)_{n\times n}$ be an operator in $\mathcal{FB}_{n}(\Omega)$, where $\tilde{T}$ satisfies the Condition (A).  
Then there exists a bounded operator in the form of $
\setlength{\arraycolsep}{1.2pt}
K=\begin{pmatrix}
0&K_{1,2}&K_{1,3}&\cdots&K_{1,n}\\
0&0&K_{2,3}&\cdots&K_{2,n}\\
\vdots&\ddots&\ddots&\ddots&\vdots\\
0&\cdots&0&0&K_{n-1,n}\\
0&\cdots&\cdots&0&0
\end{pmatrix}
$, where $K_{i,j}=\gamma_{i,j}(T_{i,i})T_{i,i+1}\cdots T_{j-1,j}(j\neq n,\gamma_{i,j}(T_{i,i})\in\{T_{i,i}\}'),K_{i,n}=T_{i,n}(1\le i\le n-1) $, such that $X=\textbf{I}+K$ is invertible and $XT=\widetilde{T}X$.
\end{lem}

\begin{proof}
The proof is by induction for $n$. The validity of the case $n=3$ is immediate from Lemma \ref{X.}. We assume that for two operators $T_{n-1}$ and $\widetilde{T}_{n-1}$ in $\mathcal{FB}_{n-1}(\Omega)$, where
$$\footnotesize{
\setlength{\arraycolsep}{1.2pt}
T_{n-1}=\begin{pmatrix}
T_{2,2}&T_{2,3}&0&\cdots&0\\
0&T_{3,3}&T_{3,4}&\cdots&0\\
\vdots&\ddots&\ddots&\ddots&\vdots\\
0&\cdots&0&T_{n-1,n-1}&T_{n-1,n}\\
0&\cdots&\cdots&0&T_{n,n}
\end{pmatrix},  \widetilde{T}_{n-1}=\begin{pmatrix}
T_{2,2}&T_{2,3}&\cdots&T_{2,n}\\
0&T_{3,3}&\cdots&T_{3,n}\\
\vdots&\ddots&\ddots&\vdots\\
0&\cdots&T_{n-1,n-1}&T_{n-1,n}\\
0&\cdots&0&T_{n,n}
\end{pmatrix},}$$
and $\tilde{T}_{n-1}$ satisfies the Condition (A),
there exists a bounded operator in the form of $K'=\begin{pmatrix}
0&K_{2,3}&K_{2,4}&\cdots&T_{2,n}\\
0&0&K_{3,4}&\cdots&T_{3,n}\\
\vdots&\ddots&\ddots&\ddots&\vdots\\
0&\cdots&0&0&T_{n-1,n}\\
0&\cdots&\cdots&0&0
\end{pmatrix}
$, where $K_{i,j}=\gamma_{i,j}(T_{i,i})T_{i,i+1}\cdots T_{j-1,j}(j\neq n,\gamma_{i,j}(T_{i,i})\in\{T_{i,i}\}')$, such that $X'=\textbf{I}+K'$ is invertible and $X'T_{n-1}=\widetilde{T}_{n-1}X'$.

Now, the lemma will be proved if we can show the result is ture in the case of $n$. Let us write the two operators $T,\ \widetilde{T}$ in the form of $2\times 2$ block matrix:
$$T=\begin{pmatrix}
T_{1,1}&T_{1\times n-1}\\
0&T_{n-1}
\end{pmatrix},
\widetilde{T}=\begin{pmatrix}
T_{1,1}&\widetilde{T}_{1\times n-1}\\
0&\widetilde{T}_{n-1}
\end{pmatrix}\in\mathcal{FB}_{n}(\Omega),$$
where $T,\ \tilde{T}$ satisfy the Condition (A)
and $T_{1\times n-1}=(T_{1,2},0,\cdots,0)$,\
$\widetilde{T}_{1\times n-1}=(T_{1,2},T_{1,3},\cdots,$
$T_{1,n})$.

Set $X:=\textbf{I}+K$, where$$K=\begin{pmatrix}
0&K_{1,2}&K_{1,3}&\cdots&K_{1,n}\\
0&0&K_{2,3}&\cdots&T_{2,n}\\
\vdots&\ddots&\ddots&\ddots&\vdots\\
0&\cdots&\cdots&0&T_{n-1,n}\\
0&\cdots&\cdots&0&0
\end{pmatrix}.$$
And the operator $K$ is also written in the form of $2\times 2$ block matrix:$$K=\begin{pmatrix}
0&K_{1\times n-1}\\
0&K'
\end{pmatrix},$$
where $K_{1\times n-1}=(K_{1,2},K_{1,3},\cdots,K_{1,n})$.
From the relation $XT=\widetilde{T}X$, we get$$\begin{pmatrix}
\textbf{I}&K_{1\times n-1}\\
0&X'
\end{pmatrix}
\begin{pmatrix}
T_{1,1}&T_{1\times n-1}\\
0&T_{n-1}
\end{pmatrix}
=\begin{pmatrix}
T_{1,1}&\widetilde{T}_{1\times n-1}\\
0&\widetilde{T}_{n-1}
\end{pmatrix}
\begin{pmatrix}
\textbf{I}&K_{1\times n-1}\\
0&X'
\end{pmatrix},
$$
which is equivalent to
\begin{equation}\label{equation3.3}
T_{1\times n-1}+K_{1\times n-1}T_{n-1}=T_{1,1}K_{1\times n-1}+\widetilde{T}_{1\times n-1}X'.
\end{equation}
For $3\le i\le n-1$, by equating the i-th entry of equation (\ref{equation3.3}), we have
\begin{equation}
\scriptsize{
\left\{
\begin{aligned}
T_{1,2}+K_{1,2}T_{2,2}&=T_{1,1}K_{1,2}+T_{1,2},\\
K_{1,2}T_{2,3}+K_{1,3}T_{3,3}&=T_{1,1}K_{1,3}+T_{1,2}K_{2,3}+T_{1,3},\\
\vdots\\
K_{1,i-1}T_{i-1,i}+K_{1,i}T_{i,i}&=T_{1,1}K_{1,i}+T_{1,2}K_{2,i}+\cdots+T_{1,i-1}K_{i-1,i}+T_{1,i},\\
\vdots\\
K_{1,n-2}T_{n-2,n-1}+K_{1,n-1}T_{n-1,n-1}&=T_{1,1}K_{1,n-1}+T_{1,2}K_{2,n-1}+\cdots+T_{1,n-2}K_{n-2,n-1}+T_{1,n-1},\\
K_{1,n-1}T_{n-1,n}+K_{1,n}T_{n,n}&=T_{1,1}K_{1,n}+T_{1,2}T_{2,n}+\cdots+T_{1,n-1}T_{n-1,n}+T_{1,n}.
\end{aligned}
\right.}
\end{equation}

Let $K_{1,n}=T_{1,n}$, then we can see  $K_{1,n}T_{n,n}=T_{1,1}K_{1,n}$ by a direct calculation. Thus we have
\begin{equation*}
\begin{aligned}
K_{1,n-1}&=T_{1,2}\phi_{2,n}(T_{2,2})T_{2,3}\cdots T_{n-2,n-1}+\cdots+\phi_{1,n}(T_{1,1})T_{1,2}\cdots T_{n-2,n-1}\\
&=\widetilde{\gamma}_{1,n-1}(T_{1,1})T_{1,2}\cdots T_{n-2,n-1},
\end{aligned}
\end{equation*}
which satisfies $K_{1,n-1}T_{n-1,n-1}=T_{1,1}K_{1,n-1}$. Furthermore, we obtain
\begin{equation*}
\begin{aligned}
K_{1,n-2}&=T_{1,2}\gamma_{2,n-1}(T_{2,2})T_{2,3}\cdots T_{n-3,n-2}+\cdots+\phi_{1,n-1}(T_{1,1})T_{1,2}\cdots T_{n-3,n-2}\\
&=\widetilde{\gamma}_{1,n-2}(T_{1,1})T_{1,2}\cdots T_{n-3,n-2},
\end{aligned}
\end{equation*}
which satisfies $K_{1,n-2}T_{n-2.n-2}=T_{1,1}K_{1,n-2}$. By following previous steps, suppose that for $3\le i\le n$, we have solved $K_{1,i}$ which satisfies $K_{1,i}T_{i,i}=T_{1,1}K_{1,i}$. Then by a simple calculation, we have
\begin{equation*}
\begin{aligned}
K_{1,i-1}&=T_{1,2}\gamma_{2,i}(T_{2,2})T_{2,3}\cdots T_{i-2,i-1}+\cdots+\phi_{1,i}(T_{1,1})T_{1,2}\cdots T_{i-2,i-1}\\
&=\widetilde{\gamma}_{1,i-1}(T_{1,1})T_{1,2}\cdots T_{i-2,i-1},
\end{aligned}
\end{equation*}
which satisfies $K_{1,i-1}T_{i-1,i-1}=T_{1,1}K_{1,i-1}$.
This completes the proof.
\end{proof}

\begin{lem}\label{mainremark}
An operator $T=\big((T_{i,j})\big)_{n\times n}$ in $\mathcal{CFB}_{n}(\Omega)$ is similar to an operator $\tilde{T}=\big((\tilde{T}_{i,j})\big)_{n\times n}$ in $\mathcal{OFB}_{n}(\Omega)$, where $\tilde{T}_{i,i}=T_{i,i}$, $1\le i \le n$, and $\tilde{T}_{i,i+1}=T_{i,i+1}$, $1\le i\le n-1$.
\end{lem}
\begin{proof}
It follows from Lemma \ref{mainlemma} directly.
\end{proof}

\textbf{The proof of Theorem \ref{main 1}: }
\begin{proof}
Following that $\mathcal{CFB}_{n}(\Omega)$ is norm dense in $\mathcal{B}_{n}(\Omega)$, there is a sequence of operators $\{\tilde{T}_{n}\}_{n=1}^{\infty}$ in $\mathcal{CFB}_{n}(\Omega)$ satisfying $$\Vert T-\tilde{T}_{n}\Vert\rightarrow0.$$
By Lemma \ref{mainremark}, there exist a sequence of operators $\{T_{n}\}_{n=1}^{\infty}$ in $\mathcal{OFB}_{n}(\Omega)$ and a sequence of bounded and invertible operators $\{X_{n}\}_{n=1}^{\infty}$ such that $$\tilde{T}_{n}=X_{n}T_{n}X_{n}^{-1},\ n\ge1.$$
Then we can see that
$$\Vert T-X_{n}T_{n}X_{n}^{-1}\Vert\rightarrow0.$$
\end{proof}

\section{Similarity of operators in $\mathcal{FB}_{n}(\Omega)$}
When we attach the Condition (A) to operators in $\mathcal{FB}_{n}(\Omega)$, the similarity problem of these operators can be translated to the similarity problem of operators in $\mathcal{OFB}_{n}(\Omega)$. Following that, we give a set of complete similarity invariants for a large class of operators in $\mathcal{FB}_{n}(\Omega)$. To begin with, we will make some preparations for our main result of this section.

In \cite{Agler.J.2}, J. Agler introduced the concept of $n-$hypercontraction operator. We call $T\in\mathcal{L}(\mathcal{H})$ a hypercontraction of order $n$ if it holds
$$\sum\limits_{j=0}^{k}(-1)^{j}\binom{k}{j}(T^{*})^{j}T^{j}\ge0,$$
for all $1\le k\le n$. And for $1\le k\le n$,
$$D_{k,T}=\left(\sum\limits_{j=0}^{k}(-1)^{j}\binom{k}{j}(T^{*})^{j}T^{j}\right)^{\frac{1}{2}},$$
are called as the defect operators. Then the following result was proved.
\begin{lem}\cite{Agler.J.2}\label{isometry}
Suppose that $T\in\mathcal{L}(\mathcal{H})$ is a $n-$hypercontraction operator. Let $\mathcal{K}_{n}$ be the Hilbert space of sequence $\{f_{k}\}_{k=0}^{\infty}$ of vectors from $\mathcal{H}$ with
$$\Vert\{f_{k}\}_{k=0}^{\infty}\Vert^{2}=\sum\limits_{k=0}^{\infty}\Vert f_{k}\Vert^{2}\binom{n+k-1}{k}<\infty,$$
and $W_{n}:\mathcal{H}\rightarrow\mathcal{K}_{n}$ be defined by
$$W_{n}x:=\{D_{n,T}T^{k}x\}_{k=0}^{\infty},\ x\in\mathcal{H}.$$
Then $W_{n}:\mathcal{H}\rightarrow\mathcal{K}_{n}$ is an isometry.
\end{lem}

Let $n$ be a positive integer. Denote  by $\mathcal{M}_{n}$ the Hilbert space of analytic functions $f=\sum\limits_{k=0}^{\infty}\hat{f}(k)z^{k}$ on the unit disc $\mathbb{D}$ satisfying $$\Vert f\Vert^{2}_{n}:=\sum\limits_{k=0}^{\infty}|\hat{f}(k)|^{2}\frac{1}{\binom{n+k-1}{k}}<\infty.$$
Note that $\mathcal{M}_{1}$ is the Hardy space $H^{2}$ and, for each positive integer $n\ge 2$, the space $\mathcal{M}_{n}$ is the weighted Bergman space. Moreover, $\mathcal{M}_{n}$ is a reproducing kernel Hilbert space with reproducing kernel $K(z,w)=\frac{1}{(1-z\bar{w})^{n}},\ z,w\in\mathbb{D}$.

Similarly, we can define the vector-valued spaces $\mathcal{M}_{n,E}$ taking values in a separable Hilbert space $E$. And on the space $\mathcal{M}_{n,E}$, there is the forward shift operator $S_{n,E},\ S_{n,E}f(z):=zf(z)$, and the backward shift operator $S^{*}_{n,E}$, its adjoint.

\begin{lem}\cite{DHS}\label{H.S}
For a $n-$hypercontraction operator $T\in\mathcal{B}_{m}(\mathbb{D})$, $T\sim_{s}S^{*}_{n,\mathbb{C}^{m}}$ if and only if there exists a bounded subharmonic function $\phi$ defined on $\mathbb{D}$ such that $$\mathcal{K}_{T}-\mathcal{K}_{S^{*}_{n,\mathbb{C}^{m}}}=\frac{\partial^{2}}{\partial\bar{w}\partial w}\phi.$$
\end{lem}

\begin{lem}\cite{NKN}\ [Model Theorem]
Every contraction $T$ on the Hilbert space $\mathcal{H}$ with property that $\lim\limits_{n\rightarrow\infty}\Vert T^{n}h\Vert=0$ for every $h\in\mathcal{H}$ is unitarily equivalent to $M^{*}_{z,E}|_{K}$ for some Hilbert space $E$ and an $M^{*}_{z,E}-$invariant subspace $K$ of $H^{2}_{E}$, where $M^{*}_{z}$ is the adjoint of the multiplication operator on the Hardy space $H^{2}$ and $H^{2}_{E}$ is vector-valued Hardy class $H^{2}$ with values in $E$.
\end{lem}
According to the model theorem, if $T\in\mathcal{B}_{1}(\mathbb{D})$ is a $n-$hypercontraction operator, then there exists a Hilbert space $E$ such that $T\sim_{u}M^{*}_{z,E}|_{K}$. In the language of  holomorphic vector bundles: there exists some holomorphic vector bundle $\mathcal{E}$ such that $E_{T}\sim_{u}E_{M^{*}_{z}}\otimes\mathcal{E}$. The concrete proof is given when we prove the main conclusion.

Clearly, an eigenvector of $T$ is an eigenvector of $M^{*}_{z}$, and the eigenspace $\ker(M^{*}_{z}-w)$ of $M^{*}_{z}$ in the Hardy space $H^{2}$ is spanned by the reproducing kernel $K(\cdot,\bar{w})$, where recall $K(z,w)=\frac{1}{1-z\bar{w}}$. So in the case of the backward shift $M^{*}_{z,E}$ in $H^{2}_{E}$,$$\ker(M^{*}_{z,E}-w)=\{K(z,\bar{w})e:\ e\in E\}.$$
Thus, the eigenspaces of $T$ are given by
$$\ker(T-w)=\{K(z,\bar{w})e:\ e\in E(w)\},$$where $E(w)$ are some subspaces of the space $E$. The vector valued Hardy space $H^{2}_{E}$ is a natural realization of the tensor product $H^{2}\otimes E$, so we can write$$\ker(T-w)=K(z,\bar{w})\otimes E(w).$$

Based on Lemma \ref{mainlemma}, we can obtain the following lemmas.

\begin{lem}\label{J.J.K.1}
Let $T$ and $\tilde{T}$ be operators in $\mathcal{FB}_{n}(\Omega)$ that satisfy the Condition (A). Let $X$ be a bounded and invertible linear operator such that
$XT=\tilde{T}X$. Then two operators $T_{1}$ and
$\tilde{T}_{1}$ in $\mathcal{OFB}_{n}(\Omega)$ whose non-zero elements are same with $T$ and $\tilde{T}$, respectively, are similar.
\end{lem}
\begin{proof}
By Lemma \ref{mainlemma}, there exist two bounded and invertible operators $Y$ and $Z$ such that
$$YT=T_{1}Y\ \ \mbox{and}\ Z\tilde{T}=\tilde{T}_{1}Z.$$
Then we can see that
$$T_{1}=YTY^{-1}=YX^{-1}\tilde{T}XY^{-1}=YX^{-1}Z^{-1}\tilde{T}_{1}ZXY^{-1},$$
that is, $T_{1}$ is similar to $\tilde{T}_{1}$.
\end{proof}

\begin{lem}\label{maincor}
Let $T$ and $\tilde{T}$ be operators in $\mathcal{OFB}_{n}(\Omega)$. If $T$ is similar to $\tilde{T}$, then two operators in $\mathcal{FB}_{n}(\Omega)$ that satisfy the Condition (A) and the elements of the primary and secondary diagonals are the same as those of $T$ and $\tilde{T}$, respectively, are similar.
\end{lem}
\begin{proof}
Let $T_{1}$ and $\tilde{T}_{1}$ be two operators in the $\mathcal{FB}_{n}(\Omega)$ whose elements on the primary and secondary diagonals are same with that of $T$ and $\tilde{T}$, respectively, and $T_{1}$, $\tilde{T}_{1}$ satisfy the Condition (A). Following $T$ is similar to $\tilde{T}$, there exists a bounded and invertible operator $X$ such that $XT=\tilde{T}X$. By Lemma \ref{mainlemma}, we can find two bounded and invertible operators $Y$ and $Z$ satisfying $YT=T_{1}Y$ and $Z\tilde{T}=\tilde{T}_{1}Z$. Then we can see  $$T_{1}=YTY^{-1}=YX^{-1}\tilde{T}XY^{-1}=YX^{-1}Z^{-1}\tilde{T}_{1}ZXY^{-1},$$
that is, $T_{1}$ is similar to $\tilde{T}_{1}$.
\end{proof}

With the above conclusions, we can reduce the similarity problem of operators in  $\mathcal{FB}_{n}(\Omega)$ to that of operators in $\mathcal{OFB}_{n}(\Omega)$. The following theorem is our similar classification for a large class of operators with a flag structure.
\begin{thm}\label{main2}
Let $T=\big((T_{i,j})\big)_{n\times n}$ and $\tilde{T}=\big((\tilde{T}_{i,j})\big)_{n\times n}$ be operators in $\mathcal{FB}_{n}(\mathbb{D})$, where $T$ and $\tilde{T}$ satisfy the Condition (A),\
$\tilde{T}_{i,i}=(M_{z}^{*},\mathcal{H}_{K_{i}})$ and
$K_{i}(z,w)=\frac{1}{(1-z\bar{w})^{\lambda_{i}}}$ for some $\lambda_{i}\in\mathbb{N}$ and for all $1\le i\le n$. Suppose that following statements hold:
\begin{enumerate}
\item Each $T_{i,i}\in\mathcal{L}(\mathcal{H}_{i})$ is a $m_{i}$-hypercontraction for $1\le i\le n$;

\item There exists a sequence of bounded and invertible holomorphic functions on $\mathbb{D}$, denoted by $\{\phi_{i}\}_{i=1}^{n-1}$, such that for all $1\le i\le n-1$ and $w\in\mathbb{D}$,
$$|\phi_{i}(w)|^{2}\frac{\Vert T_{i,i+1}t_{i+1}(w)\Vert ^{2}}{\Vert t_{i+1}(w)\Vert ^{2}}=\frac{\Vert \tilde{T}_{i,i+1}K_{i+1}(\cdot,\bar{w})\Vert ^{2}}{\Vert K_{i+1}(\cdot,\bar{w})\Vert ^{2}},$$
where $t_{i}(w)\in \ker (T_{i,i}-w)$ and $t_{i}(w)=T_{i,i+1}t_{i+1}(w),\ K_{i}(\cdot,\bar{w})=\tilde{T}_{i,i+1}K_{i+1}(\cdot,\bar{w})$.
\end{enumerate}
Then $T\sim_{s}\tilde{T}$ if and only if $\mathcal{K}_{\tilde{T}_{n,n}}-\mathcal{K}_{T_{n,n}}=\frac{\partial^{2}}{\partial\bar{w}\partial w}\psi$, for some bounded subharmonic function $\psi$ defined on $\mathbb{D}$.
\end{thm}

\begin{proof}
Based on Lemma \ref{J.J.K.1} and \ref{maincor}, all we have to do is consider whether the conclusion is true when $T_{i,j},\ \tilde{T}_{i,j}=0,\ j-i\ge2$.
We first consider the necessity and assume that $XT=\tilde{T}X$ for some bounded and invertible operator $X$. Then by Lemma \ref{J.J.K.M}, $X$ and $X^{-1}$ are upper-triangular, and $X_{i,i}$, $1\le i\le n$, is invertible. So we have
$X_{i,i}T_{i,i}=\tilde{T}_{i,i}X_{i,i},\ 1\le i\le n$.
Now, since $T_{n,n}$ is a $m_{n}$-hypercontraction, by Lemma \ref{H.S}, there exists a bounded subharmonic function $\psi$ defined on $\mathbb{D}$ such that $$\mathcal{K}_{\tilde{T}_{n,n}}-\mathcal{K}_{T_{n,n}}=\frac{\partial^{2}}{\partial\bar{w}\partial w}\psi.$$

Conversely, we start by defining the operators $U_{i}:\mathcal{H}_{i}\rightarrow\mathcal{M}_{i}$ by $$U_{i}x=\sum\limits_{n=0}^{\infty}\frac{z^{n}}{\Vert z^{n}\Vert^{2}_{i}}\otimes D_{m_{i},T_{i,i}}T_{i,i}^{n}x,\ \mbox{for\ all}\ x\in\mathcal{H}_{i},$$ where $D_{m_{i},T_{i,i}}=\left(\sum\limits_{k=0}^{m_{i}}(-1)^{k}\binom{m_{i}}{k}(T_{i,i}^{*})^{k}T_{i,i}^{k}\right)^{\frac{1}{2}}$ are the defect operators,\ $\mathcal{M}_{i}:=\overline{ran U_{i}}$, and $\Vert z^{n}\Vert_{i}$ denotes the norm of $z^{n}$ on the space $\mathcal{H}_{K_{i}}$.

By using the result of Lemma \ref{isometry}, we calculate the norm of $U_{i}x$ for all $1\le i\le n$ and we obtain
\begin{equation*}
\begin{aligned}
\Vert U_{i}x\Vert^{2}
&=\langle\sum\limits_{n=0}^{\infty}\frac{z^{n}}{\Vert z^{n}\Vert^{2}_{i}}\otimes D_{m_{i},T_{i,i}}T_{i,i}^{n}x,\ \sum\limits_{m=0}^{\infty}\frac{z^{m}}{\Vert z^{m}\Vert^{2}_{i}}\otimes D_{m_{i},T_{i,i}}T_{i,i}^{m}x\rangle\\
\end{aligned}
\end{equation*}
\begin{equation*}
\begin{aligned}
\ \ \ \ \ \ \ &=\sum\limits_{n=0}^{\infty}\sum\limits_{m=0}^{\infty}\langle\frac{z^{n}}{\Vert z^{n}\Vert^{2}_{i}}\otimes D_{m_{i},T_{i,i}}T_{i,i}^{n}x,\ \frac{z^{m}}{\Vert z^{m}\Vert^{2}_{i}}\otimes D_{m_{i},T_{i,i}}T_{i,i}^{m}x\rangle\\
&=\sum\limits_{n=0}^{\infty}\frac{1}{\Vert z^{n}\Vert^{2}_{i}}\langle D_{m_{i},T_{i,i}}T_{i,i}^{n}x,\ D_{m_{i},T_{i,i}}T_{i,i}^{n}x\rangle\\
&=\Vert x\Vert^{2}_{i}.
\end{aligned}
\end{equation*}
It follows that each $U_{i}$ is an unitary operator and satisfies $U_{i}T_{i,i}=M^{*}_{z}|_{\mathcal{M}_{i}}U_{i}$, $1\le i\le n$, by a simple calculation.
Then follwing $t_{i}(w)\in \ker(T_{i,i}-w)$ and $t_{i}(w)=T_{i,i+1}t_{i+1}(w)$ for $w\in\mathbb{D}$, we can obtain
\begin{equation*}
\begin{aligned}
U_{i}t_{i}(w)&=\sum\limits_{n=0}^{\infty}\frac{z^{n}}{\Vert z^{n}\Vert^{2}_{i}}\otimes D_{m_{i},T_{i,i}}T_{i,i}^{n}t_{i}(w)\\
&=\sum\limits_{n=0}^{\infty}\frac{z^{n}w^{n}}{\Vert z^{n}\Vert^{2}_{i}}\otimes D_{m_{i},T_{i,i}}t_{i}(w)\\
&=K_{i}(z,\bar{w})\otimes D_{m_{i},T_{i,i}}t_{i}(w),
\end{aligned}
\end{equation*}
for $1\le i\le n$. Denote $\chi(w):=D_{m_{1},T_{1,1}}t_{1}(w)\in\mathcal{H}_{1}.$

Next we define the operators $V_{i}$ as
\begin{equation*}
V_{1}t_{1}(w)=K_{1}(\cdot,\bar{w})\otimes \chi(w),\ w\in\mathbb{D},
\end{equation*}
\begin{equation*}
V_{i}t_{i}(w)=\prod\limits_{k=1}^{i-1}\phi_{k}(w)K_{i}(\cdot,\bar{w})\otimes \chi(w),\ w\in\mathbb{D},\ 2\le i \le n.
\end{equation*}
By a calculation, we can see that
\begin{equation*}
\begin{aligned}
\Vert t_{1}(w)\Vert^{2}
&=\Vert K_{1}(\cdot,\bar{w})\otimes D_{m_{1},T_{1,1}}t_{1}(w)\Vert^{2}\\
&=\Vert K_{1}(\cdot,\bar{w})\otimes \chi(w)\Vert^{2}\\
&=\Vert K_{1}(\cdot,\bar{w})\Vert^{2}\Vert \chi(w)\Vert^{2}\\
&=\Vert V_{1}t_{1}(w)\Vert^{2}.
\end{aligned}
\end{equation*}
According to $$|\phi_{i}(w)|^{2}\frac{\Vert T_{i,i+1}t_{i+1}(w)\Vert ^{2}}{\Vert t_{i+1}(w)\Vert ^{2}}=\frac{\Vert \tilde{T}_{i,i+1}K_{i+1}(\cdot,\bar{w})\Vert ^{2}}{\Vert K_{i+1}(\cdot,\bar{w})\Vert ^{2}},$$
for the bounded and invertible holomorphic function $\phi_{i}$, $1\le i\le n-1$, we get  that for $2\le i\le n$,
$$\Vert t_{i}(w)\Vert^{2}=\prod\limits_{k=1}^{i-1}|\phi_{k}(w)|^{2}K_{i}(\bar{w},\bar{w})\Vert \chi(w)\Vert^{2}=\Vert V_{i}t_{i}(w)\Vert^{2}.$$
Therefore, the operators $V_{i}$, $1\le i\le n$, are isometries. Moreover, setting $\mathcal{N}_{i}=\overline{\mbox{ran} V_{i}}$, the isometries $V_{i}\in\mathcal{L}(\mathcal{H}_{i},\mathcal{N}_{i})$ become unitary operators and
\begin{equation}
\tiny{
\begin{pmatrix}
V_{1}&0&\cdots&0&0\\
0&V_{2}&\cdots&0&0\\
\vdots&\ddots&\ddots&\vdots&\vdots\\
0&\cdots&\cdots&V_{n-1}&0\\
0&\cdots&\cdots&0&V_{n}
\end{pmatrix}
\begin{pmatrix}
T_{1,1}&T_{1,2}&0&\cdots&0\\
0&T_{2,2}&T_{2,3}&\cdots&0\\
\vdots&\ddots&\ddots&\ddots&\vdots\\
0&\cdots&\cdots&T_{n-1,n-1}&T_{n-1,n}\\
0&\cdots&\cdots&0&T_{n,n}
\end{pmatrix}
\begin{pmatrix}
V_{1}^{*}&0&\cdots&0&0\\
0&V_{2}^{*}&0\cdots&0\\
\vdots&\ddots&\ddots&\vdots&\vdots\\
0&\cdots&\cdots&V_{n-1}^{*}&0\\
0&\cdots&\cdots&0&V_{n}^{*}
\end{pmatrix}}
\end{equation}

\begin{equation*}
\footnotesize{
\begin{aligned}
&=\begin{pmatrix}
V_{1}U_{1}^{*}M_{z}^{*}|_{\mathcal{M}_{1}}U_{1}V_{1}^{*}&V_{1}T_{1,2}V_{2}^{*}&0&\cdots&0\\
0&V_{2}U_{2}^{*}M_{z}^{*}|_{\mathcal{M}_{2}}U_{2}V_{2}^{*}&V_{2}T_{2,3}V_{3}^{*}&\cdots&0\\
\vdots&\ddots&\ddots&\ddots&\vdots\\
0&\cdots&\cdots&0&V_{n}U_{n}^{*}M_{z}^{*}|_{\mathcal{M}_{n}}U_{n}V_{n}^{*}
\end{pmatrix}\\
&=\begin{pmatrix}
M_{z}^{*}|_{\mathcal{N}_{1}}&V_{1}T_{1,2}V_{2}^{*}&0&\cdots&0\\
0&M_{z}^{*}|_{\mathcal{N}_{2}}&V_{2}T_{2,3}V_{3}^{*}&\cdots&0\\
\vdots&\ddots&\ddots&\ddots&\vdots\\
0&\cdots&\cdots&M_{z}^{*}|_{\mathcal{N}_{n-1}}&V_{n-1}T_{n-1,n}V_{n}^{*}\\
0&\cdots&\cdots&0&M_{z}^{*}|_{\mathcal{N}_{n}}
\end{pmatrix}.
\end{aligned}}
\end{equation*}
From this, we deduce that $$T_{i,i}\sim_{u}M_{z}^{*}|_{\mathcal{N}_{i}}.$$

In addition, for all $1\le i\le n-1$,
$$\frac{\partial^{2}}{\partial\bar{w}\partial w}\log\Bigg(|\phi_{i}(w)|^{2}\frac{\Vert T_{i,i+1}t_{i+1}(w)\Vert ^{2}}{\Vert \tilde{T}_{i,i+1}K_{i+1}(\cdot,\bar{w})\Vert ^{2}}\Bigg)=\frac{\partial^{2}}{\partial\bar{w}\partial w}\log\frac{\Vert t_{i+1}(w)\Vert ^{2}}{\Vert K_{i+1}(\cdot,\bar{w})\Vert ^{2}}$$
implies $$\mathcal{K}_{\tilde{T}_{i,i}}-\mathcal{K}_{T_{i,i}}=\mathcal{K}_{\tilde{T}_{i+1,i+1}}-\mathcal{K}_{T_{i+1,i+1}}.$$
Furthermore,
$$\mathcal{K}_{\tilde{T}_{1,1}}-\mathcal{K}_{T_{1,1}}=\mathcal{K}_{\tilde{T}_{2,2}}-\mathcal{K}_{T_{2,2}}=\cdots=\mathcal{K}_{\tilde{T}_{n,n}}-\mathcal{K}_{T_{n,n}}.$$

Since  $T_{i,i}\sim_{u}M_{z}^{*}|_{\mathcal{N}_{i}}$, $1\le i\le n$, we have
$$\mathcal{K}_{\tilde{T}_{i,i}}-\mathcal{K}_{M_{z}^{*}|_{\mathcal{N}_{i}}}=
\mathcal{K}_{\tilde{T}_{i,i}}-\mathcal{K}_{T_{i,i}}=
\mathcal{K}_{\tilde{T}_{n,n}}-\mathcal{K}_{T_{n,n}}
=\frac{\partial^{2}}{\partial\bar{w}\partial w}\psi.$$
Under this condition, it is shown that there exist bounded and invertible operators $X_{i}\in\mathcal{L}(\mathcal{H}_{K_{i}},\mathcal{N}_{i})$ such that $X_{i}\tilde{T}_{i,i}=M_{z}^{*}|_{\mathcal{N}_{i}}X_{i}$ for $1\le i \le n$. Then it follows for $1\le i\le n$ and $w\in\mathbb{D}$ that $X_{i}K_{i}(\cdot,\bar{w})\in \ker(M^{*}_{z}|_{\mathcal{N}_{i}}-w)$ by a simple calculation. According to the fact that for $2\le i\le n$ and $w\in\mathbb{D}$,
$$\ker(M_{z}^{*}|_{\mathcal{N}_{1}}-w)=\bigvee\limits_{w\in\mathbb{D}}K_{1}(\cdot,\bar{w})\otimes \chi(w)\ \mbox{and}\  \ker(M_{z}^{*}|_{\mathcal{N}_{i}}-w)=\bigvee\limits_{w\in\mathbb{D}}\prod\limits_{k=1}^{i-1}\phi_{k}(w)K_{i}(\cdot,\bar{w})\otimes \chi(w),$$it follows that there exists some $\lambda(w)\in\mbox{Hol}(\mathbb{D})$ such that
$$X_{1}K_{1}(\cdot,\bar{w})=\lambda(w)K_{1}(\cdot,\bar{w})\otimes \chi(w),$$
$$X_{i}K_{i}(\cdot,\bar{w})=\lambda(w)\prod\limits_{k=1}^{i-1}\phi_{k}(w)K_{i}(\cdot,\bar{w})\otimes \chi(w),\ 2\le i\le n,$$
for every $w\in\mathbb{D}$.

It can easily be seen that $$X=\begin{pmatrix}
X_{1}&0&0&\cdots&0\\
0&X_{2}&0&\cdots&0\\
\vdots&\ddots&\ddots&\ddots&\vdots\\
0&\cdots&\cdots&X_{n-1}&0\\
0&\cdots&\cdots&0&X_{n}
\end{pmatrix}$$ is a bounded and invertible operator. So to prove that $T$ is similar to $\tilde{T}$, we just have to verify $X_{i}\tilde{T}_{i,i+1}$ $=V_{i}T_{i,i+1}V^{*}_{i+1}X_{i+1}$, $1\le i\le n-1$.
By a direct calculation, we have
\begin{equation*}
\begin{aligned}
X_{1}\tilde{T}_{1,2}K_{2}(\cdot,\bar{w})
&=X_{1}K_{1}(\cdot,\bar{w})\\
&=\lambda(w)K_{1}(\cdot,\bar{w})\otimes \chi(w)\\
&=V_{1}(\lambda(w)t_{1}(w))\\
&=V_{1}T_{1,2}(\lambda(w)t_{2}(w))\\
&=V_{1}T_{1,2}V^{*}_{2}(\lambda(w)\phi_{1}(w)K_{1}(\cdot,\bar{w})\otimes \chi(w))\\
&=V_{1}T_{1,2}V^{*}_{2}X_{2}K_{2}(\cdot,\bar{w}),
\end{aligned}
\end{equation*}
which means $X_{1}\tilde{T}_{1,2}=V_{1}T_{1,2}V^{*}_{2}X_{2}.$
Similarly, for $2\le i\le n-1$, we have
\begin{equation*}
\begin{aligned}
V_{i}T_{i,i+1}V_{i+1}^{*}X_{i+1}K_{i+1}(\cdot,\bar{w})
&=V_{i}T_{i,i+1}V_{i+1}^{*}(\lambda(w)\prod\limits_{k=1}^{i}\phi_{k}(w)K_{i+1}(\cdot,\bar{w})\otimes \chi(w))\\
&=V_{i}T_{i,i+1}(\lambda(w)t_{i+1}(w))\\
&=\lambda(w)\prod\limits_{k=1}^{i-1}\phi_{k}(w)K_{i}(\cdot,\bar{w})\otimes \chi(w)\\
&=X_{i}K_{i}(\cdot,\bar{w})\\
&=X_{i}\tilde{T}_{i,i+1}K_{i+1}(\cdot,\bar{w}),
\end{aligned}
\end{equation*}
that is, $$X_{i}\tilde{T}_{i,i+1}=V_{i}T_{i,i+1}V_{i+1}^{*}X_{i+1}.$$
The proof is completed.
\end{proof}

\section{Application to weakly homogeneous operators}

In this section, we study a class of weakly homogenous operators in the class $\mathcal{FB}_{n}(\mathbb{D})$.

Let M$\ddot{o}$b denote the group of all biholomorphic automorphisms of the unit disc $\mathbb{D}$ and $\bar{\mathbb{D}}$ denote the unit closed disc. An operator $T$ in $\mathcal{L}(\mathcal{H})$ is said to be weakly homogeneous if $\sigma(T)\subseteq\bar{\mathbb{D}}$ and $\phi(T)$ is similar to $T$ for all $\phi$ in $\mbox{M}\ddot{o}\mbox{b}$(see \cite{BM1}).
It is easy to verify that $T$ is weakly homogeneous if and only if $T^{*}$ is weakly homogeneous.

A function $g\in\mathcal{H}$ introduces a linear operator $C_{g}:\mathcal{H}\rightarrow\mathcal{H}$, defined by
$$C_{g}(f)=f\circ g,\ \mbox{for any}\ f\in\mathcal{H}.$$
Then the operator $C_{g}$ is called to be a composition operator. The boundedness of the composition operators with symbols of M$\ddot{o}$bius transformations plays an important role in the problem of weak homogeneity of multiplication operators $M_{z}$. If the composition operator $C_{\phi}$ is bounded for all $\phi\in\mbox{M$\ddot{o}$b}$, then the bounded multiplication operator $M_{z}$ is a weakly homogeneous operator. Therefore we can study the boundedness of composition operator $C_{\phi}$, $ \phi\in\mbox{M$\ddot{o}$b}$, which leads us to the conclusion of weak homogeneity of the bounded multiplication operator $M_{z}$. However, we should notice that, in general, the weak homogeneity of the multiplication operator $M_{z}$ fails to show the boundedness of composition operator $C_{\phi}$ for all $ \phi\in\mbox{M$\ddot{o}$b}$.

Recall that a Hilbert space $\mathcal{H}$ consisting of holomorphic functions on the unit disc $\mathbb{D}$ is said to be M$\ddot{o}$bius invariant if for each $\phi\in\mbox{M}\ddot{o}\mbox{b}$, $f\circ\phi\in\mathcal{H}$ whenever $f\in\mathcal{H}$.
By applying the closed graph theorem, it can be concluded that $\mathcal{H}$ is M$\ddot{o}$bius invariant if and only if the composition operator $C_{\phi}$ is bounded on $\mathcal{H}$ for each $\phi\in\mbox{M}\ddot{o}\mbox{b}$. Consequently, if the multiplication operator $M_{z}$ is bounded on some M$\ddot{o}$bius invariant Hilbert space $\mathcal{H}$, then $M_{z}$ is weakly homogeneous on $\mathcal{H}$.

Denote by $C(\bar{\mathbb{D}})$ the space of all continuous functions on $\bar{\mathbb{D}}$. If $\psi$ is an arbitrary function in $C(\bar{\mathbb{D}})\cap \mbox{Hol}(\mathbb{D})$, then it is easy to see that $\psi\in\mbox{H}^{\infty}(\mathbb{D})$.
By Lemma \ref{mainlemma}, we obtain the main result of this section as follows, which generalizes the Proposition 5.4 in \cite{JJ} with a different proof.

\begin{thm}\label{main1}
Let $M_{i,z}$ be the bounded multiplication operator on the M$\ddot{o}$bius invariant Hilbert space $\mathcal{H}_{K_{i}}$, where   $K_{i}(z,w)=\sum\limits_{k=0}^{\infty}a_{i,k}z^{k}\bar{w}^{k},\ a_{i,k}>0,\ z,w\in\mathbb{D},\ 1\le i\le n$ and $\lim\limits_{k\rightarrow\infty}k\sqrt{\frac{a_{i,k}}{a_{i+1,k}}}=\infty$ for $1\le i\le n-1$. Suppose that $\mathcal{H}_{K_{1}}\subseteq\mathcal{H}_{K_{2}}\subseteq\cdots\subseteq\mathcal{H}_{K_{n}}$ and $\mbox{Mult}(\mathcal{H}_{K_{i}})=\mbox{H}^{\infty}(\mathbb{D})$ for $1\le i\le n$.

Set$$T:=\begin{pmatrix}
M_{1,z}^{*}&M_{\psi_{1,2}}^{*}&T_{1,3}&\cdots&T_{1,n-1}&T_{1,n}\\
0&M_{2,z}^{*}&M_{\psi_{2,3}}^{*}&T_{2,4}&\cdots&T_{2,n}\\
\vdots&\ddots&\ddots&\ddots&\ddots&\vdots\\
0&\cdots&\cdots&M_{n-2,z}^{*}&M_{\psi_{n-2,n-1}}^{*}&T_{n-2,n}\\
0&\cdots&\cdots&0&M_{n-1,z}^{*}&M_{\psi_{n-1,n}}^{*}\\
0&\cdots&\cdots&0&0&M_{n,z}^{*}
\end{pmatrix},$$where
$T$ satisfies the Condition (A) and
$\psi_{i,i+1}\in C(\bar{\mathbb{D}})\cap \mbox{Hol}(\mathbb{D}), 1\le i\le n-1$, is non-zero. The operator $T$ is weakly homogeneous if and only if each $\psi_{i,i+1}, 1\le i\le n-1$, is non-vanishing.
\end{thm}
\begin{rem}\label{remark}
Since $\psi_{i,i+1}\in C(\bar{\mathbb{D}})\cap \mbox{Hol}(\mathbb{D})$, $i=1,2,\cdots,n-1$, and $\mbox{Mult}(\mathcal{H}_{K_{i}})=\mbox{H}^{\infty}(\mathbb{D})$, $i=1,2,\cdots,n$, each $M_{\psi_{i,i+1}}:\mathcal{H}_{K_{i}}\rightarrow\mathcal{H}_{K_{i}}$ is a bounded operator. Consequently, by Theorem 6.25 in  \cite{Paulsen}, $M_{\psi_{i,i+1}}:\mathcal{H}_{K_{i}}\rightarrow\mathcal{H}_{K_{i+1}}$, $1\le i\le n-1$, is bounded. 
\end{rem}

\begin{lem}\label{property H}
Assume $M_{i,z}$ be the multiplication operator on the reproducing kernel Hilbert space $\mathcal{H}_{K_{i}}$, where $K_{i}(z,w)=\sum\limits_{n=0}^{\infty}a_{i,n}z^{n}\bar{w}^{n},\ a_{i,n}>0,\ z,w\in\mathbb{D},\ i=1,2$. If $\lim\limits_{n\rightarrow\infty}n\sqrt{\frac{a_{1,n}}{a_{2,n}}}=\infty$,
then $M^{*}_{1,z}$ and $M^{*}_{2,z}$ satisfy the Property (H).
\end{lem}
\begin{proof}
As we all known that $M_{i,z}^{*}$ can be regarded as the weighted unilateral shift with the weighted sequence $\{\sqrt{\frac{a_{i,n}}{a_{i,n+1}}}\}_{n=0}^{\infty}$. Then using the Proposition 3.5 in \cite{JJ}, the conclusion is true.
\end{proof}
Next, we give the proof of the main theorem.

\textbf{The proof of Theorem \ref{main1}: }
\begin{proof}
By Lemma \ref{mainlemma}, we have known that $T$ is weakly homogeneous if and only if $T_{1}$ is weakly homogeneous, where $$T_{1}=\begin{pmatrix}
M_{1,z}^{*}&M_{\psi_{1,2}}^{*}&0&\cdots&0&0\\
0&M_{2,z}^{*}&M_{\psi_{2,3}}^{*}&0&\cdots&0\\
\vdots&\ddots&\ddots&\ddots&\ddots&\vdots\\
0&\cdots&\cdots&M_{n-2,z}^{*}&M_{\psi_{n-2,n-1}}^{*}&0\\
0&\cdots&\cdots&0&M_{n-1,z}^{*}&M_{\psi_{n-1,n}}^{*}\\
0&\cdots&\cdots&0&0&M_{n,z}^{*}
\end{pmatrix}.$$
We first assume that $T_{1}$ is weakly homogeneous, which is equivalent to $T_{1}^{*}$ is weakly homogeneous.
By a routine calculation, we obtain
$$\tiny{
\setlength{\arraycolsep}{1.2pt}
\phi(T_{1})=\begin{pmatrix}
\phi(M_{1,z}^{*})&\phi'(M_{1,z}^{*})M_{\psi_{1,2}}^{*}&\frac{1}{2}\phi''(M_{1,z}^{*})M_{\psi_{1,2}\psi_{2,3}}^{*}&*&\cdots&*&*\\
0&\phi(M_{2,z}^{*})&\phi'(M_{2,z}^{*})M_{\psi_{2,3}}^{*}&\frac{1}{2}\phi''(M_{2,z}^{*})M_{\psi_{2,3}\psi_{3,4}}^{*}&\cdots&*&*\\
0&0&\phi(M_{3,z}^{*})&\phi'(M_{3,z}^{*})M _{\psi_{3,4}}^{*}&\cdots&*&*\\
\vdots&\ddots&\ddots&\ddots&\ddots&\vdots&\vdots\\
0&0&\cdots&\cdots&\cdots&\phi(M_{n-1,z}^{*})&\phi'(M_{n-1,z}^{*})M_{\psi_{n-1,n}}^{*}\\
0&0&\cdots&\cdots&\cdots&0&\phi(M_{n,z}^{*})
\end{pmatrix},}$$
for each $\phi$ in $\mbox{M}\ddot{o}\mbox{b}$.
By Lemma \ref{mainlemma}, it follows that $\phi(T_{1})$ is similar to $T'$, where $$T'=\begin{pmatrix}
\phi(M_{1,z}^{*})&\phi'(M_{1,z}^{*})M_{\psi_{1,2}}^{*}&0&\cdots&0&0\\
0&\phi(M_{2,z}^{*})&\phi'(M_{2,z}^{*})M_{\psi_{2,3}}^{*}&\cdots&0&0\\
0&0&\phi(M_{3,z}^{*})&\cdots&0&0\\
\vdots&\ddots&\ddots&\ddots&\ddots&\vdots\\
0&\cdots&\cdots&0&\phi(M_{n-1,z}^{*})&\phi'(M_{n-1,z}^{*})M_{\psi_{n-1,n}}^{*}\\
0&\cdots&\cdots&0&0&\phi(M_{n,z}^{*})
\end{pmatrix}.$$
Thus we have $T_{1}\sim_{s}T'$. Then by Lemma 2.6 and 4.6 in \cite{JJ},
there exist invertible and bounded operators $X_{1},\cdots,X_{n}$ such that $X_{i}M_{i,z}^{*}=\phi(M_{i,z}^{*})X_{i},\ 1\le i\le n$, and $X_{i}M_{\psi_{i,i+1}}^{*}=\phi'(M_{i,z}^{*})M_{\psi_{i,i+1}}^{*}X_{i+1}$, $1\le i \le n-1.$
Following $X_{i}M_{i,z}^{*}=\phi(M_{i,z}^{*})X_{i},\ 1\le i\le n$, there exists a $\psi_{i}\in H^{\infty}(\mathbb{D})$ such that $X_{i}=M_{\psi_{i}}C_{\phi}$ where $X_{i}(1):=\psi_{i}\in\mathcal{H}_{K_{i}}$. Then we can see that $X_{i}(K_{i,\bar{w}}(\cdot))=\psi_{i}(w)K_{i,\overline{\phi(w)}}(\cdot)$ for all $w$ in $\mathbb{D}$, $1\le i\le n$. From $X_{i}M_{\psi_{i,i+1}}^{*}=\phi'(M_{i,z}^{*})M_{\psi_{i,i+1}}^{*}X_{i+1}$, $1\le i \le n-1$, we obtain
\begin{equation}\label{5.1}
X_{i}M_{\psi_{i,i+1}}^{*}K_{i+1,\bar{w}}(\cdot)=\phi'(M_{i,z}^{*})M_{\psi_{i,i+1}}^{*}X_{i+1}K_{i+1,\bar{w}}(\cdot),\ w\in\mathbb{D}.
\end{equation}
By calculating the left side of the equation (\ref{5.1}), we get
$$X_{i}M_{\psi_{i,i+1}}^{*}K_{i+1,\bar{w}}(\cdot)=X_{i}K_{i,\bar{w}}(\cdot)\overline{\psi_{i,i+1}(\bar{w})}=\psi_{i}(w)\overline{\psi_{i,i+1}(\bar{w})}K_{i,\overline{\phi(w)}}(\cdot),\  w\in\mathbb{D},$$
where the first equal sign is because for all $h\in\mathcal{H}_{K_{i}}$,
\begin{equation*}
\langle M_{\psi_{i,i+1}}^{*}K_{i+1,\bar{w}}(\cdot),h\rangle=\langle K_{i+1,\bar{w}}(\cdot),\psi_{i,i+1}(\cdot)h(\cdot)\rangle=\overline{\psi_{i,i+1}(\bar{w})h(\bar{w})}
=\langle K_{i,\bar{w}}(\cdot)\overline{\psi_{i,i+1}(\bar{w})},h\rangle.
\end{equation*}
Similarly, calculate the other side of the equation (\ref{5.1}) and we have
\begin{equation*}
\begin{aligned}
\phi'(M_{i,z}^{*})M_{\psi_{i,i+1}}^{*}X_{i+1}K_{i+1,\bar{w}}(\cdot)&=\phi'(M_{i,z}^{*})\psi_{i+1}(w)\overline{\psi_{i,i+1}(\overline{\phi(w)})}K_{i,\overline{\phi(w)}}(\cdot)\\
&=\psi_{i+1}(w)\overline{\psi_{i,i+1}(\overline{\phi(w)})}\phi'(\phi(w))K_{i,\overline{\phi(w)}}(\cdot),w\in\mathbb{D}.
\end{aligned}
\end{equation*}
Here the second equal sign is because for all $h\in\mathcal{H}_{K_{i}}$,
\begin{equation*}
\langle\phi'(M_{i,z}^{*})K_{i,\bar{w}}(\cdot),h\rangle
=\langle K_{i,\bar{w}}(\cdot),\big(\phi'(M_{i,z}^{*})\big)^{*}h(\cdot)\rangle
=\langle K_{i,\bar{w}}(\cdot),\hat{\phi'}(\cdot)h(\cdot)\rangle
=\langle K_{i,\bar{w}}(\cdot)\overline{\hat{\phi'}(\bar{w})},h\rangle,
\end{equation*}
where $\hat{\phi'}(\bar{w})=\overline{\phi'(w)}$. So by comparing the two sides of the equation (\ref{5.1}), we obtain
\begin{equation}\label{5.2}
\psi_{i}(w)\overline{\psi_{i,i+1}(\bar{w})}=\psi_{i+1}(w)\overline{\psi_{i,i+1}(\overline{\phi(w)})}\phi'(\phi(w)), w\in\mathbb{D}.
\end{equation}

Next, we will start with equation (\ref{5.2}) and prove by contradiction that $\psi_{i,i+1},1\le i\le n-1$, is non-vanishing in $\mathbb{D}$. Suppose that there exists a point $w_{0}\in\mathbb{D}$ such that $\psi_{i,i+1}(\overline{w_{0}})=0$. Note that $\phi'(\phi(w_{0}))\ne0$. If the statement would not hold, that is, $\phi'(\phi(w_{0}))=0$, then we can obtain $w_{0}=\frac{1}{\bar{\alpha}}$ is not in $\mathbb{D}$ if we assume $\phi(w)=\frac{\alpha-w}{1-\bar{\alpha}w},\ |\alpha|<1$. Meanwhile, for $1\le i\le n$, $X_{i}$ is bounded and invertible, so $\psi_{i}(w)\neq 0$ for all $w\in\mathbb{D}$.
Therefore following the fact that M$\ddot{o}$bius acts transitively on $\mathbb{D}$, one can see that $\psi_{i,i+1}(w)=0$ for all $w$ in $\mathbb{D}$ and hence $\psi_{i,i+1}(w)=0$ for all $w$ in $\overline{\mathbb{D}}$. This contradicts the fact that $\psi_{i,i+1}$ is a non-zero function. Therefore $\psi_{i,i+1}(w)\ne 0$ for all $w$ in $\mathbb{D}$.

Now we prove that $\psi_{i,i+1}, 1\le i\le n-1$, is non-vanishing on $\mathbb{T}=\{z\in\mathbb{C}:|z|=1\}$. Replacing $\phi$ by biholomorphic map $z\mapsto e^{i\theta}z$
in equation (\ref{5.2}), we have
\begin{equation}\label{5.3}
\psi_{i}(w)\overline{\psi_{i,i+1}(\bar{w})}=e^{i\theta}\psi_{i+1}(w)\overline{\psi_{i,i+1}(e^{-i\theta}\bar{w})},\ w\in\mathbb{D},\ \theta\in\mathbb{R}.
\end{equation}
Suppose there exists a point $e^{i\theta_{0}}$ such that $\psi_{i,i+1}(e^{i\theta_{0}})=0$. Choose a sequence $\{w_{n}\}$ in $\mathbb{D}$ such that $w_{n}\rightarrow e^{-i\theta_{0}}$ as $n\rightarrow \infty$. From the equation (\ref{5.3}), we get
\begin{equation}\label{5.4}
\psi_{i}(w_{n})\overline{\psi_{i,i+1}(\overline{w_{n}})}=e^{i\theta}\psi_{i+1}(w_{n})\overline{\psi_{i,i+1}(e^{-i\theta}\overline{w_{n}})}.
\end{equation}
Since $\psi_{i,i+1}\in C(\bar{\mathbb{D}})$ and $\psi_{i},\ \psi_{i+1}$ are bounded above and below on $\mathbb{D}$, so from the equation (\ref{5.4}), as $n\rightarrow\infty$, we obtain $\psi_{i,i+1}(e^{i(\theta_{0}-\theta)})=0$ for all $\theta\in\mathbb{R}$. Thus $\psi_{i,i+1}$ is zero at every point of $\mathbb{T}$, and then by an application of  maximum modulus principle, it follows that $\psi_{i,i+1}$ is identically zero on $\bar{\mathbb{D}}$, which is a contraction to our assumption that $\psi_{i,i+1}$ is a non-zero function.

Conversely, it suffices to show that $T_{1}^{*}$ is weakly homogeneous. By Lemma \ref{mainlemma}, we just need to prove that $T_{1}^{*}\sim_{s}T_{2}$, where
$$T_{2}=\begin{pmatrix}
M_{\phi}&0&0&\cdots&0&0\\
M_{\psi_{1,2}\phi'}&M_{\phi}&0&\cdots&0&0\\
0&M_{\psi_{2,3}\phi'}&M_{\phi}&\cdots&0&0\\
\vdots&\ddots&\ddots&\ddots&\vdots&\vdots\\
\vdots&\ddots&\ddots&\ddots&\ddots&\vdots\\
0&0&\cdots&0&M_{\psi_{n-1,n}\phi'}&M_{\phi}
\end{pmatrix}.
$$
Since each $\psi_{i,i+1},\ 1\le i\le n-1$, is in $\mbox{C}(\bar{\mathbb{D}})\cap\mbox{Hol}(\mathbb{D})$ and is non-vanishing,\ $\frac{1}{\psi_{i,i+1}}$ is a bounded analytic function on $\mathbb{D}$ for $1\le i\le n-1$. Following the condition that $\mbox{Mult}(\mathcal{H}_{K_{i}})=\mbox{H}^{\infty}(\mathbb{D})$, $1\le i \le n$, we can see that $M_{\psi_{i,i+1}}$ is bounded and invertible on the each Hilbert space $\mathcal{H}_{K_{i}}$. Note that for each $\phi\in\mbox{M}\ddot{o}\mbox{b}$ and $1\le i\le n$, $C_{\phi}$ is bounded and invertible on the $\mbox{M}\ddot{o}\mbox{bius}$ invariant Hilbert space $\mathcal{H}_{K_{i}}$.
Set $$X=\begin{pmatrix}
X_{1}&0&0&\cdots&0\\
0&X_{2}&0&\cdots&0\\
\vdots&\ddots&\ddots&\ddots&\vdots\\
0&\cdots&0&X_{n-1}&0\\
0&\cdots&\cdots&0&X_{n}
\end{pmatrix},$$
where $X_{1}=M_{\frac{(\psi_{1,2}\circ\phi)(\phi'\circ\phi)}{\psi_{1,2}}}\cdots M_{\frac{(\psi_{n-1,n}\circ\phi)(\phi'\circ\phi)}{\psi_{n-1,n}}}C_{\phi },\
X_{2}=M_{\frac{(\psi_{2,3}\circ\phi)(\phi '\circ\phi)}{\psi_{2,3}}}\cdots \notag M_{\frac{(\psi_{n-1,n}\circ\phi )(\phi '\circ\phi)}{\psi_{n-1,n}}}C_{\phi},\\
\cdots,\ X_{n}=C_{\phi}.$ Since each $X_{i}$, $1\le i\le n$, is invertible,
$X$ is invertible and satisfies $XT_{2}=T_{1}^{*}X$. The proof of the theorem is now complete.
\end{proof}

\end{document}